\documentclass[
  11pt,
  reqno
]{amsart}

\usepackage{mathtools}
\usepackage{amssymb}
\usepackage{amsthm}
\usepackage{mathrsfs}

\usepackage[no-math]{fontspec}

\usepackage{microtype}
\usepackage{enumitem}
\usepackage{tikz}

\usepackage[
  backend=biber,
  style=numeric,
  doi=true,
  sorting=nyt,
  eprint=true,
  giveninits=true,
  maxnames=99
]{biblatex}
\DeclareFieldFormat[article]{title}{\mkbibemph{#1}}
\renewbibmacro*{in:}{%
  \ifentrytype{article}
    {}
    {\printtext{\bibstring{in}\intitlepunct}}}

\usepackage[
  colorlinks=false,
  linkbordercolor={1 0 0},
  citebordercolor={0 1 0},
  urlbordercolor={0 1 1}
]{hyperref}
\usepackage[nameinlink,noabbrev]{cleveref}

\DeclarePairedDelimiter{\abs}{\lvert}{\rvert}

\DeclarePairedDelimiter{\set}{\lbrace}{\rbrace}

\DeclareMathOperator{\diam}{diam}

\DeclareMathOperator{\cl}{cl}
\DeclareMathOperator{\Int}{Int}

\DeclareMathOperator{\Met}{Met}
\DeclareMathOperator{\CPM}{CPM}
\DeclareMathOperator{\Comp}{Comp}
\DeclareMathOperator{\Tr}{Tr}
\DeclareMathOperator{\WF}{WF}

\newcommand{\UMet}[2]{\mathrm{UMet}(#1;#2)}
\newcommand{\UCPM}[2]{\mathrm{UCPM}(#1;#2)}
\newcommand{\UComp}[2]{\mathrm{UComp}(#1;#2)}
\newcommand{\UMetbd}[2]{\mathrm{UMet}_{\mathrm{bd}}(#1;#2)}
\newcommand{\UCPMbd}[2]{\mathrm{UCPM}_{\mathrm{bd}}(#1;#2)}
\newcommand{\UCompbd}[2]{\mathrm{UComp}_{\mathrm{bd}}(#1;#2)}
\newcommand{\Compbd}[1]{\mathrm{Comp}_{\mathrm{bd}}(#1)}
\DeclareMathOperator{\umetdis}{\mathcal{UD}}

\newcommand{\yoomega}{\omega_{0}}
\newcommand{\Baire}{{}^{\yoomega}\yoomega}
\newcommand{\Nseq}{\mathcal{S}}

\newcommand{\TreeBoundarySpace}{\mathring{\Upsilon}}

\newcommand{\yov}[1]{v_{#1}}
\newcommand{\yoc}[1]{c(#1)}
\newcommand{\yoa}[2]{w_{#1}(#2)}
\newcommand{\yod}[1]{d_{#1}}
\newcommand{\yotreeNu}{\mu}
\newcommand{\yotreeext}[1]{\widehat{#1}}
\newcommand{\yotreerho}{\rho}

\newcommand{\yocmpc}[1]{c(#1)}
\newcommand{\yocmpa}[2]{w_{#1}(#2)}
\newcommand{\yocmpd}[1]{d_{#1}}
\newcommand{\yocmpext}[1]{\widehat{#1}}
\newcommand{\yocmprho}{\rho_r}

\newcommand{\youcompc}[1]{c(#1)}
\newcommand{\youcompa}[2]{w_{#1}(#2)}
\newcommand{\youcompd}[1]{d_{#1}}

\newcommand{\youext}[1]{\widehat{#1}}
\newcommand{\yourho}{\rho_r}
\newcommand{\youc}[1]{c(#1)}
\newcommand{\youa}[2]{w_{#1}(#2)}
\newcommand{\youu}[1]{u_{#1}}
\newcommand{\restr}{\mathord{\upharpoonright}}

\newcommand{\yometdis}[1]{\mathcal{D}_{#1}}

\newcommand{\yometbd}[1]{\mathrm{BMet}(#1)}
\newcommand{\yocpmbd}[1]{\mathrm{BCPM}(#1)}

\newcommand{\yopcl}{\mathfrak{P}}
\newcommand{\yopclq}[1]{\yopcl_{#1}}

\newcommand{\youpclq}[1]{\mathfrak{P}_{#1}^{R}}

\newcommand{\yoccl}{\mathfrak{C}}
\newcommand{\yocclq}[1]{\yoccl(#1)}
\newcommand{\youccl}{\mathfrak{C}^{R}}
\newcommand{\youcclq}[1]{\youccl(#1)}

\numberwithin{equation}{section}

\theoremstyle{plain}
\newtheorem{theorem}{Theorem}[section]
\newtheorem{lemma}[theorem]{Lemma}
\newtheorem{proposition}[theorem]{Proposition}
\newtheorem{corollary}[theorem]{Corollary}

\theoremstyle{definition}
\newtheorem{definition}[theorem]{Definition}

\theoremstyle{remark}
\newtheorem{remark}[theorem]{Remark}

\newenvironment{acknowledgements}{%
  \medskip\noindent\textit{Acknowledgements.}\ }{\par}
\newenvironment{useofai}{%
  \medskip\noindent\textbf{Use of AI.}\ }{\par}

\crefname{theorem}{theorem}{theorems}
\Crefname{theorem}{Theorem}{Theorems}

\crefname{lemma}{lemma}{lemmas}
\Crefname{lemma}{Lemma}{Lemmas}

\crefname{proposition}{proposition}{propositions}
\Crefname{proposition}{Proposition}{Propositions}

\crefname{corollary}{corollary}{corollaries}
\Crefname{corollary}{Corollary}{Corollaries}

\crefname{claim}{claim}{claims}
\Crefname{claim}{Claim}{Claims}

\crefname{definition}{definition}{definitions}
\Crefname{definition}{Definition}{Definitions}

\crefname{example}{example}{examples}
\Crefname{example}{Example}{Examples}

\crefname{question}{question}{questions}
\Crefname{question}{Question}{Questions}

\crefname{remark}{remark}{remarks}
\Crefname{remark}{Remark}{Remarks}
\title[Borel moduli spaces of metrics]
{If the Moduli Space of Metrics Is Borel, Then It Is Completely Metrizable}

\author{Yoshito Ishiki}
\address{Department of Mathematical Sciences\\
Tokyo Metropolitan University\\
Minami-osawa, Hachioji, Tokyo 192-0397, Japan}
\email{ishiki-yoshito@tmu.ac.jp}

\date{\today}



\begin{document}

\begin{abstract}
Let $X$ be a metrizable space. We prove that the space of compatible metrics on $X$ is Borel in the space of continuous pseudometrics, equipped with the topology of uniform convergence, if and only if $X$ is $\sigma$-compact. These conditions are equivalent to complete metrizability of the space of compatible metrics; the same characterization holds for each component. For ultrametrizable $X$ and $R\subseteq[0,\infty)$ containing $0$ and arbitrarily small positive elements, we prove analogous results for compatible $R$-valued ultrametrics and their $R$-components, {with neighborhoods defined by agreement at every point pair where either ultrametric exceeds a positive $\epsilon\in R$}. For the nonempty subspaces of complete compatible metrics and complete compatible $R$-valued ultrametrics, Borelness and complete metrizability are equivalent to compactness of $X$. The proofs use continuous families of metrics parametrized by trees or subsets of compact completions.
\end{abstract}

\maketitle

\section{Introduction}
Let $X$ be a metrizable space with a fixed topology. We write $\Met(X)$ for the metrics generating this topology and $\CPM(X)$ for the continuous pseudometrics on $X$. Throughout the metric part of the paper, we regard $\Met(X)$ as a subspace of $\CPM(X)$ with the topology of uniform convergence.

For $d,e\in\CPM(X)$, set
\[
 \yometdis{X}(d,e)
 =\sup_{(x,y)\in X^2}\abs{d(x,y)-e(x,y)}.
\]
The supremum may be infinite. This extended distance induces the topology of uniform convergence.

{In the 2020 preprint version (arXiv:2003.13227v1) of} \cite[proof of Lemma~5.1]{ishiki-interpolation},
the author proved that
$\Met(X)$
is completely metrizable when
$X$
is second-countable and locally compact.
Koshino extended this sufficient condition to
$\sigma$-compact metrizable spaces for bounded compatible metrics
\cite[Proposition~3]{koshino-topologies}.
His argument also applies to
$\Met(X)$,
as noted in \cite[Section~1.2]{ishiki-baire}.
Koshino subsequently proved that,
for separable metrizable
$X$,
complete metrizability of
$\Met(X)$
is equivalent to
$\sigma$-compactness of
$X$
\cite[Corollary~1.2 and Remark~1.7]{koshino-borel}.

Ishiki and Uda proved that Borelness of
$\Met(X)$
in
$\CPM(X)$
forces separability of
$X$
\cite[Theorem~3.7]{ishiki-uda-borel}.
Combining this result with Koshino's characterization,
they removed the separability assumption from the equivalence between complete metrizability of
$\Met(X)$
and
$\sigma$-compactness of
$X$
\cite[Corollary~3.9]{ishiki-uda-borel}.
It remains to ask whether Borelness of
$\Met(X)$
alone forces
$\sigma$-compactness of
$X$.

In the present paper,
we prove that Borelness of
$\Met(X)$
in
$\CPM(X)$
is equivalent both to
$\sigma$-compactness of
$X$
and to complete metrizability of
$\Met(X)$.

\begin{theorem}\label{thm:introduction}
Let
$X$
be a metrizable space.
Then the following conditions are equivalent.
\begin{enumerate}[label=\textup{(M\arabic*)},ref=\textup{(M\arabic*)}]
\item\label{item:intro-met-borel}
$\Met(X)$
is Borel in
$\CPM(X)$.
\item\label{item:intro-met-sigma}
$X$
is
$\sigma$-compact.
\item\label{item:intro-met-complete}
$\Met(X)$
is completely metrizable.
\end{enumerate}
\end{theorem}


\Cref{thm:introduction} provides a descriptive set theoretic justification for assuming
$\sigma$-compactness of
$X$
when studying
$\Met(X)$.
Indeed,
this assumption is necessary as soon as
$\Met(X)$
is required to be Borel in
$\CPM(X)$,
and it also ensures complete metrizability of
$\Met(X)$.

We also characterize Borelness in each component. For $d\in\Met(X)$, set
\[
 \yocclq{d}=\set*{e\in\Met(X)\middle|\yometdis{X}(e,d)<\infty}.
\]
Ishiki and Uda proved that $\yocclq{d}$ is completely metrizable if and only if $X$ is $\sigma$-compact \cite[Corollary~3.9]{ishiki-uda-borel}. The following theorem includes Borelness in this equivalence.

\begin{theorem}\label{thm:introduction-components}
Let
$X$
be metrizable and let
$d\in\Met(X)$.
Then the set
$\yocclq{d}$
is
both the connected component and the path component of
$d$
in
$\Met(X)$.
Moreover,
the following conditions are equivalent.
\begin{enumerate}[label=\textup{(C\arabic*)},ref=\textup{(C\arabic*)}]
\item\label{item:intro-metric-component-borel}
$\yocclq{d}$
is Borel in
$\CPM(X)$.
\item\label{item:intro-metric-component-sigma}
$X$
is
$\sigma$-compact.
\item\label{item:intro-metric-component-complete}
$\yocclq{d}$
is completely metrizable.
\end{enumerate}
Thus,
if
$X$
is not
$\sigma$-compact,
every component of
$\Met(X)$
is
non-Borel.
\end{theorem}

For complete compatible metrics, compactness replaces $\sigma$-compactness. Write
\[
 \Comp(X)=\set*{d\in\Met(X)\middle|(X,d)\text{ is complete}}.
\]
When $X$ is not completely metrizable, this space is empty and hence both Borel and completely metrizable.

\begin{theorem}\label{thm:introduction-comp}
Let
$X$
be a metrizable space.
Then the following conditions are equivalent.
\begin{enumerate}[label=\textup{(K\arabic*)},ref=\textup{(K\arabic*)}]
\item\label{item:intro-complete-metric-ambient}
$\Comp(X)$
is Borel in
$\CPM(X)$.
\item\label{item:intro-complete-metric-relative}
$\Comp(X)$
is Borel relative to
$\Met(X)$.
\item\label{item:intro-complete-metric-complete}
$\Comp(X)$
is completely metrizable.
\item\label{item:intro-complete-metric-compact}
either
$\Comp(X)$
is empty or
$X$
is compact.
\end{enumerate}
In particular,
if
$\Comp(X)$
is nonempty,
these conditions hold exactly when
$X$
is compact.
\end{theorem}

We next consider ultrametrics with a prescribed range. A \emph{range set} is a set $R\subseteq[0,\infty)$ containing $0$. It is \emph{characteristic} if it contains positive elements arbitrarily close to $0$. Let $R$ be characteristic and let $X$ be ultrametrizable. Denote the continuous $R$-valued pseudo-ultrametrics by $\UCPM{X}{R}$ and the compatible $R$-valued ultrametrics by $\UMet{X}{R}$. Write $r\lor s=\max\{r,s\}$, and interpret the maximum of a function and a constant pointwise. For $d,e\in\UCPM{X}{R}$, define
\[
 \umetdis_X^R(d,e)
 =\inf\set*{\epsilon\in R\setminus\set{0}\middle|
 d\leq e\lor\epsilon\ \text{and}\ e\leq d\lor\epsilon
 \text{ on }X^2},
\]
with $\inf\emptyset=\infty$. This is the extended ultrametric of \cite[Remark~1.8]{ishiki-embedding}. We equip $\UCPM{X}{R}$ and its subspaces with the topology it induces. For $u\in\UMet{X}{R}$, set
\[
 \begin{aligned}
 \youpclq{u}
 &=\set*{p\in\UCPM{X}{R}\middle|\umetdis_X^R(p,u)<\infty},\\
 \youcclq{u}&=\youpclq{u}\cap\UMet{X}{R}.
 \end{aligned}
\]
The set $\youcclq{u}$ is called the \emph{$R$-component} of $u$. Thus $R$-components are the equivalence classes in $\UMet{X}{R}$ under finiteness of $\umetdis_X^R$.

For this topology, the following analogue of \Cref{thm:introduction} holds.

\begin{theorem}\label{thm:introduction-ultrametric}
Let
$R$
be characteristic and let
$X$
be ultrametrizable.
Then the following
conditions are equivalent.
\begin{enumerate}[label=\textup{(U\arabic*)},ref=\textup{(U\arabic*)}]
\item\label{item:intro-ultrametric-borel}
$\UMet{X}{R}$
is Borel in
$\UCPM{X}{R}$.
\item\label{item:intro-ultrametric-sigma}
$X$
is
$\sigma$-compact.
\item\label{item:intro-ultrametric-complete}
$\UMet{X}{R}$
is completely metrizable in the topology induced by
$\umetdis_X^R$.
\end{enumerate}
\end{theorem}

The same conditions characterize Borelness and complete metrizability of each $R$-component.

\begin{theorem}\label{thm:introduction-ultrametric-components}
Let
$R$
be characteristic,
let
$X$
be ultrametrizable,
and let
$u\in\UMet{X}{R}$.
Then the following conditions are equivalent.
\begin{enumerate}[label=\textup{(F\arabic*)},ref=\textup{(F\arabic*)}]
\item\label{item:intro-ultrametric-component-borel}
$\youcclq{u}$
is Borel in
$\UCPM{X}{R}$.
\item\label{item:intro-ultrametric-component-sigma}
$X$
is
$\sigma$-compact.
\item\label{item:intro-ultrametric-component-complete}
$\youcclq{u}$
is completely metrizable in the topology induced by
$\umetdis_X^R$.
\end{enumerate}
Thus,
if
$X$
is not
$\sigma$-compact,
every $R$-component
is
non-Borel.
\end{theorem}


Write $\UComp{X}{R}$ for the complete members of $\UMet{X}{R}$. As for ordinary metrics, compactness characterizes Borelness and complete metrizability of this subspace whenever it is nonempty.


\begin{theorem}\label{thm:introduction-ucomp}
Let
$R$
be characteristic and let
$X$
be ultrametrizable.
Then the following
conditions are equivalent.
\begin{enumerate}[label=\textup{(R\arabic*)},ref=\textup{(R\arabic*)}]
\item\label{item:intro-complete-ultrametric-ambient}
$\UComp{X}{R}$
is Borel in
$\UCPM{X}{R}$.
\item\label{item:intro-complete-ultrametric-relative}
$\UComp{X}{R}$
is Borel relative to
$\UMet{X}{R}$.
\item\label{item:intro-complete-ultrametric-complete}
$\UComp{X}{R}$
is completely metrizable in the topology induced by
$\umetdis_X^R$.
\item\label{item:intro-complete-ultrametric-compact}
either
$\UComp{X}{R}$
is empty or
$X$
is compact.
\end{enumerate}
In particular,
if
$\UComp{X}{R}$
is nonempty,
these conditions hold exactly
when
$X$
is compact.
\end{theorem}

We outline the proof of \Cref{thm:introduction}.
Assume that $\Met(X)$ is Borel in $\CPM(X)$. To prove that $X$ is $\sigma$-compact, first use \Cref{thm:borel-separable,prop:compact-completion} to embed $X$ densely in a compact metrizable space $Y$. Suppose that $X$ is not $\sigma$-compact. We distinguish whether $X$ is Borel in $Y$.

If $X$ is Borel in $Y$, the consequence of Hurewicz's theorem stated in \Cref{thm:hurewicz} provides a closed copy of the zero-dimensional Baire space. Here $\yoomega=\set{0,1,2,\ldots}$ is discrete, and its countable product $\Baire$ is the \emph{zero-dimensional Baire space}. On the topological sum of $\Baire$ and a countable discrete space, we construct a continuous family of metrics indexed by trees. A metric generates the fixed topology exactly when its tree is well-founded. Since well-founded trees form a non-Borel set, an isometric extensor carries this family to $X$ and proves non-Borelness of $\Met(X)$, as stated in \Cref{cor:closed-baire}.


If $X$ is non-Borel in $Y$, we instead construct a continuous family indexed by a compact subset of $Y$. Here we employ Koshino's construction in the proof of \cite[Theorem~1.1]{koshino-borel}. The parameters giving compatible metrics form a non-Borel set, so $\Met(X)$ is again non-Borel. Both cases contradict the initial assumption. Hence $X$ is $\sigma$-compact.

For complete metrics, we use a family on a countable discrete space. {Each metric in this family is compatible, and it is complete exactly when the corresponding tree is well-founded.} A noncompact completely metrizable space $X$ contains a closed countably infinite discrete subspace. An extensor preserving completeness transfers this family to $X$, proving that $\Comp(X)$ is non-Borel relative to $\Met(X)$.

For ultrametrics, sums are replaced by maxima, the positive values used to define the metrics are chosen in $R$, and extension uses a maximum product. Characteristicity permits a sequence of positive elements of $R$ tending to zero. If $R$ is not characteristic, every compatible $R$-valued ultrametric is uniformly discrete and hence complete.

\medskip
\noindent\textbf{Organization.}
\Cref{sec:preliminaries} collects the extension and descriptive set theoretic results. \Cref{sec:trees,sec:boundary} construct the two families used in \Cref{sec:main} to prove \Cref{thm:introduction}. \Cref{sec:components} proves the component characterization, and \Cref{sec:complete-metrics} treats complete metrics. The last two sections establish the corresponding results for prescribed range ultrametrics, {using $\umetdis_X^R$ and extensions defined by maxima}.

\medskip
\noindent\textbf{Conventions and notation.}
Borelness and complete metrizability always refer to the metrizable topologies induced by the specified extended distances. For ordinary metrics, \emph{component} means an equivalence class under finiteness of the supremum distance. By \Cref{lem:metric-components}, this is also a connected and path component. For prescribed range ultrametrics, we use \emph{$R$-component} for an equivalence class under finiteness of $\umetdis_X^R$. Our convention for the natural numbers is $\yoomega=\set{0,1,2,\ldots}$.

\begin{useofai}
Generative AI tools,
including OpenAI Codex,
were used in the preparation of
this manuscript for language editing,
\LaTeX{} typesetting assistance,
and
exploration of proof constructions.
{The author verified the arguments,
understands the proofs and constructions,
and takes full responsibility for the mathematical content and final manuscript.}
\end{useofai}

\begin{acknowledgements}
This work was supported by JSPS KAKENHI Grant Number JP24KJ0182.
\end{acknowledgements}

\section{Preliminaries}\label{sec:preliminaries}
We fix the notation for pseudometrics and components, then recall the extension results and the facts about trees and Borel sets needed in the proofs.

\subsection{Metrics and components}

Let $X$ be metrizable. Denote by $\CPM(X)$ the pseudometrics $d\colon X\times X\to[0,\infty)$ continuous for the fixed topology of $X$, and by $\Met(X)$ the metrics generating that topology. For $d,e\in\CPM(X)$, define
\begin{equation}\label{eq:uniform-distance}
 \yometdis{X}(d,e)=\sup_{(x,y)\in X^2}\abs{d(x,y)-e(x,y)}.
\end{equation}
For $d\in\Met(X)$, set
\begin{equation}\label{eq:metric-components}
 \begin{aligned}
 \yopclq{d}&=\set*{q\in\CPM(X)\middle|\yometdis{X}(q,d)<\infty},\\
 \yocclq{d}&=\yopclq{d}\cap\Met(X).
 \end{aligned}
\end{equation}
The extended distance $\yometdis{X}$ makes $\CPM(X)$ a topological sum of clopen metric spaces, its components. In particular, the induced topology is metrizable. We use this topology on $\CPM(X)$ and all its subspaces.

We denote the complete compatible metrics by
\[
 \Comp(X)=\set*{d\in\Met(X)\middle|(X,d)\text{ is complete}}.
\]
This set is nonempty precisely when $X$ is completely metrizable. As usual, the empty space is regarded as completely metrizable.

\subsection{Extension and background results}

We use the following extensor to extend continuous pseudometrics from a closed subspace $A$ of $X$ to $X$. It preserves the supremum distance and compatibility in both directions.

\begin{proposition}\label{prop:extensor}
Let
$A$
be a closed subspace of a metrizable space
$X$.
Then there is an
isometric embedding
\[
 E\colon\CPM(A)\to\CPM(X)
\]
such that,
for every
$d\in\CPM(A)$
we have
\begin{equation}\label{eq:extensor-restriction}
 E(d)\restr A^2=d.
\end{equation}
Moreover,
\begin{equation}\label{eq:extensor-compatibility}
 E(\Met(A))=E(\CPM(A))\cap\Met(X).
\end{equation}
The extensor sends every bounded continuous pseudometric on
$A$
to a bounded continuous pseudometric on
$X$.
For every topological space
$P$
and continuous
$q\colon P\to\CPM(A)$,
\begin{equation}\label{eq:extensor-reduction}
(E\circ q)^{-1}(\Met(X))=q^{-1}(\Met(A)).
\end{equation}
If
$X$
is completely metrizable,
then the extensor
$E$
may be chosen so
that,
for every
$d\in\Met(A)$,
\begin{equation}\label{eq:extensor-completeness}
 E(d)\in\Comp(X)\quad\Longleftrightarrow\quad d\in\Comp(A).
\end{equation}
\end{proposition}

\begin{proof}
The assertion about continuous pseudometrics in \cite[Theorem~1.1]{ishiki-extensor} provides an isometric extensor satisfying \eqref{eq:extensor-restriction} and \eqref{eq:extensor-compatibility}. Let $p\in\CPM(A)$ be bounded, and choose a bounded $b\in\Met(A)$. By \cite[Theorem~1.1(3)]{ishiki-extensor}, $E(b)$ is bounded. Since $E$ is isometric,
\begin{equation}\label{eq:extensor-bounded-distance}
 \yometdis{X}(E(p),E(b))=\yometdis{A}(p,b)<\infty,
\end{equation}
and therefore $E(p)$ is bounded. If $X$ is completely metrizable, the construction in \cite[Section~4, in particular Proposition~4.9]{ishiki-extensor} {can be chosen so that, for every $d\in\Met(A)$,}
\[
 d\in\Comp(A)\quad\Longrightarrow\quad E(d)\in\Comp(X).
\]
{Conversely,} if $E(d)$ is complete, its restriction to the closed subspace $A$ is complete. This proves \eqref{eq:extensor-completeness}. Finally, \eqref{eq:extensor-reduction} follows by taking inverse images in \eqref{eq:extensor-compatibility}.
\end{proof}

Borelness of the full space of compatible metrics imposes the following restriction on $X$.

\begin{theorem}[{Ishiki--Uda, \cite[Theorem~3.7]{ishiki-uda-borel}}]\label{thm:borel-separable}
Let
$X$
be a metrizable space.
If
$\Met(X)$
is Borel in
$\CPM(X)$,
then
$X$
is separable.
\end{theorem}


{For a separable metrizable space $X$, complete metrizability of $\Met(X)$ has the following characterization.}

\begin{theorem}[{Consequence of \cite[Corollary~3.9]{ishiki-uda-borel}}]\label{thm:koshino}
Let
$X$
be a separable metrizable space.
Then
$X$
is
$\sigma$-compact
if and
only if
$\Met(X)$
is completely metrizable.
\end{theorem}


\begin{proposition}[{Consequence of \cite[Theorem~4.14]{kechris}}]\label{prop:compact-completion}
Every separable metrizable space admits a totally bounded compatible metric.
In particular, it can be embedded as a dense subspace of a compact metrizable space.
\end{proposition}

\begin{theorem}[{\cite[Theorem~3.11]{kechris}}]\label{thm:complete-subspaces-gdelta}
Let
$Y$
be a metrizable space and let
$S\subseteq Y$.
If
$S$
is completely metrizable, then
$S$
is
$G_\delta$
in
$Y$.
If
$Y$
is completely metrizable and
$S$
is
$G_\delta$
in
$Y$,
then
$S$
is completely metrizable.
\end{theorem}

\subsection{Trees and Borel sets}

We use the following space of trees as a parameter space.

\begin{definition}\label{def:tree-space}
For
$n\in\yoomega$,
a sequence of length
$n$
is a function from
$\set{i\in\yoomega\mid i<n}$
to
$\yoomega$.
The sequence of length zero is the empty function.
Thus
$\Nseq=\bigcup_{n\in\yoomega}{}^n\yoomega$,
and we denote the length of
$s\in\Nseq$
by
$|s|$.
For
$s,t\in\Nseq$,
we say that
$s$
is an \emph{initial segment} of
$t$,
and write
$s\preccurlyeq t$,
if
\begin{equation}\label{eq:initial-segment}
 |s|\leq |t|\quad\text{and}\quad s(i)=t(i)\text{ for every }i<|s|.
\end{equation}
It is a \emph{proper initial segment} when also
$|s|<|t|$.
For a finite or infinite sequence
$b$
and
$n\in\yoomega$
not exceeding its length in the finite case,
$b\restr n$
denotes its restriction to
$\set{i\in\yoomega\mid i<n}$.
In particular,
$b\restr0$
is the empty sequence.
A subset
$T\subseteq\Nseq$
is a \emph{tree} if,
for all
$s,t\in\Nseq$,
the conditions
$t\in T$
and
$s\preccurlyeq t$
imply
$s\in T$.
The empty tree is allowed.
Let
$2=\set{0,1}$
have the discrete topology.
We denote by
$\Tr$
the set of all trees,
viewed as a subspace of the product
$2^{\Nseq}$
by identifying each tree with its characteristic function.
For
$T\in\Tr$,
a sequence
$b\in\Baire$
is an \emph{infinite branch} of
$T$
if
$b\restr n\in T$
for every
$n\in\yoomega$.
A tree is \emph{well-founded} if it has no infinite branch.
We write
\[
 \WF=\set*{T\in\Tr\middle|T\text{ is well-founded}}.
\]

Fix a bijection $\yotreeNu\colon\Nseq\to\yoomega$. For $E\subseteq\Nseq$, let $\chi_E\colon\Nseq\to2$ be its characteristic function. Identifying each subset with its characteristic function, {equip $2^{\Nseq}$ with the ultrametric defined, for $E,F\subseteq\Nseq$, by}
\begin{equation}\label{eq:tree-parameter-metric}
 \Delta(E,F)=\sup_{s\in\Nseq}
  2^{-\yotreeNu(s)}\abs{\chi_E(s)-\chi_F(s)}
 \qquad(E,F\subseteq\Nseq).
\end{equation}
The positive weights separate distinct subsets, and the coordinatewise strong triangle inequality passes to the supremum. For $m\in\yoomega$, the inequality $\Delta(E,F)<2^{-m}$ holds exactly when $\chi_E(s)=\chi_F(s)$ for every $s\in\Nseq$ with $\yotreeNu(s)\leq m$. These conditions fix finitely many coordinates, and every finite subset of $\Nseq$ is contained in $\set{s\in\Nseq\mid\yotreeNu(s)\leq m}$ for some $m$. Thus $\Delta$ induces the product topology. We also write $\Delta$ for its restriction to $\Tr$.
\end{definition}

The sets $\set{b\in\Baire\mid b\restr |s|=s}$, where $s\in\Nseq$, form a clopen base of the zero-dimensional Baire space. We call them \emph{cylinders}. For each $s\in\Nseq$, define its \emph{zero extension} $\yotreeext{s}\in\Baire$ by
\begin{equation}\label{eq:zero-extension}
 \yotreeext{s}(i)=
 \begin{cases}
 s(i),&i<|s|,\\
 0,&i\geq |s|.
 \end{cases}
\end{equation}

Since $\Nseq$ is countably infinite, $2^{\Nseq}$ is homeomorphic to the Cantor set. The space $\Tr$ is a closed subspace of this compact Polish space and is therefore compact Polish.

\begin{theorem}[{Consequence of \cite[Theorem~27.1]{kechris} and \cite[Example~4.2.1]{srivastava-borel}}]\label{thm:wf-nonborel}
The set
$\WF$
of well-founded trees is non-Borel in
$\Tr$.
\end{theorem}

The following consequence of Hurewicz's theorem will be applied to a space embedded in a compact completion.

\begin{theorem}[{Consequence of Hurewicz, \cite[Theorem~21.18]{kechris}}]\label{thm:hurewicz}
Let
$Y$
be a compact metrizable space and let
$S\subseteq Y$
be Borel.
If
$S$
is not
$\sigma$-compact,
then there is a Cantor set
$K\subseteq Y$
such that
$K\setminus S$
is countable and dense in
$K$.
In particular,
$K\cap S$
is closed in
$S$
and homeomorphic to the zero-dimensional Baire space
$\Baire$.
\end{theorem}


\section{A family of metrics indexed by trees}\label{sec:trees}
We construct a continuous map from trees to metrics on the space $\TreeBoundarySpace$ consisting of the vertices of a tree in which every vertex has countably infinitely many immediate successors, together with its boundary at infinity. The metric associated with $T$ will generate the topology of $\TreeBoundarySpace$ if and only if $T$ has no infinite branch.

\subsection{The space $\TreeBoundarySpace$ and its metrics}

Set $P=\Baire$. Its standard bounded complete ultrametric is
\begin{equation}\label{eq:baire-tree-distance}
 \yotreerho(x,y)=
 \begin{cases}
  0,&x=y,\\
  2^{-\min\set{n\in\yoomega\mid x(n)\ne y(n)}},&x\ne y.
 \end{cases}
\end{equation}
For $s\in\Nseq$, use its zero extension $\yotreeext{s}$ from \eqref{eq:zero-extension}.

Choose a set $V$ disjoint from $P$ and a bijection $v\colon\Nseq\to V$. We write $\yov{s}=v(s)$ and equip $V$ with the discrete topology. Our fixed space is the topological sum
\[
 \TreeBoundarySpace=P\mathbin{\sqcup}V.
\]
Thus $U\subseteq \TreeBoundarySpace$ is open exactly when $U\cap P$ is open in $P$, with no restriction on $U\cap V$. The space $\TreeBoundarySpace$ consists of the vertices $V$ of the tree of all finite sequences together with its boundary at infinity $P$, whose points are the infinite branches, as in \Cref{fig:space-a}. Distinct vertices remain distinct even when their zero extensions agree. For example, $\yov{(1)}\ne\yov{(1,0)}$, whereas $\yotreeext{(1)}=\yotreeext{(1,0)}$.

\begin{figure}[tbp]
\centering
\providecommand{\BaireSpaceFigureName}{zero-dimensional Baire space}
\providecommand{\BoundaryFigureName}{Boundary at infinity}
\begin{tikzpicture}[x=1cm,y=1cm,font=\small,
 vertex/.style={circle,fill=black,inner sep=1.45pt},
 edge/.style={draw=black!55,line width=.45pt},
 branch/.style={draw=black,line width=.95pt},
 endpt/.style={circle,draw=black,fill=white,inner sep=1.5pt}]
 \draw[line width=1.8pt] (-.2,.65)--(10.6,.65);
 \node[vertex,label=above:{$\yov{\emptyset}$}] (root) at (5.3,5.8) {};
 \node[vertex,label=above left:{$\yov{(0)}$}] (a) at (1.5,4.5) {};
 \node[vertex,label=above right:{$\yov{(1)}$}] (b) at (5.3,4.5) {};
 \node[vertex,label=above right:{$\yov{(2)}$}] (c) at (9.0,4.5) {};
 \draw[edge] (root)--(a) (root)--(c);
 \draw[branch] (root)--(b);
 \node at (10.4,4.4) {$\cdots$};
 \foreach \x/\name/\parent in {0.3/aa/a,1.5/ab/a,4.1/ba/b,5.3/bb/b,7.8/ca/c,9.0/cb/c}{
   \node[vertex] (\name) at (\x,3.2) {};
   \draw[edge] (\parent)--(\name);
 }
 \draw[branch] (b)--(ba);
 \node[above left] at (aa) {$\yov{(0,0)}$};
 \node[above right] at (ab) {$\yov{(0,1)}$};
 \node[above left] at (ba) {$\yov{(1,0)}$};
 \node[above right] at (bb) {$\yov{(1,1)}$};
 \node[above left] at (ca) {$\yov{(2,0)}$};
 \node[above right] at (cb) {$\yov{(2,1)}$};
 \foreach \parent/\x in {a/2.6,b/6.4,c/10.1}{
   \draw[edge] (\parent)--(\x,3.4);
   \node at (\x,3.2) {$\cdots$};
 }
 \foreach \parent/\x in {aa/0.3,ab/1.5,ba/4.1,bb/5.3,ca/7.8,cb/9.0}{
   \foreach \shift in {-.32,0,.32}{
     \node[vertex,inner sep=1pt] at (\x+\shift,2.25) {};
     \draw[edge] (\parent)--(\x+\shift,2.25);
     \draw[edge,densely dotted] (\x+\shift,2.25)--(\x+1.25*\shift,1.20);
     \node[endpt,inner sep=.8pt] at (\x+1.25*\shift,.65) {};
   }
 }
 \draw[branch] (ba)--(3.78,2.25);
 \draw[branch,densely dotted] (3.78,2.25)--(3.7,1.2);
 \draw[branch,densely dashed] (3.7,1.2)--(3.7,.68);
 \node[endpt] at (3.7,.65) {};
 \node[below] at (3.7,.55) {$b=(1,0,0,\ldots)$};
 \node[anchor=east] at (10.6,5.7) {$V=\{\yov{s}\mid s\in\Nseq\}$};
 \node[anchor=north,align=center] at (5.3,-.35) {\BoundaryFigureName\ $P=\Baire$\\(\BaireSpaceFigureName)};
\end{tikzpicture}
\caption{The space $\TreeBoundarySpace$ consisting of the vertices of the tree of finite sequences together with its boundary at infinity. Each vertex has countably many immediate successors. The edges indicate the initial-segment relation, and the thick horizontal line represents the boundary at infinity $P$, whose points are infinite branches. No interior points of the edges are added to $\TreeBoundarySpace$. The topology used here is $P\sqcup V$, with $V$ discrete and $P$ and $V$ clopen.}
\label{fig:space-a}
\end{figure}

Using the bijection $\yotreeNu$ fixed in \Cref{def:tree-space}, assign to $s\in\Nseq$ the weight
\begin{equation}\label{eq:tree-weights}
  \yoc{s}=2^{-\yotreeNu(s)}.
\end{equation}
Then
$c\colon\Nseq\to\set{2^{-n}\mid n\in\yoomega}=\set{1,2^{-1},2^{-2},\ldots}$
is a bijection.
Consequently, for every $\epsilon>0$, only finitely many weights are at least $\epsilon$. In particular, for every $b\in P$, the weights $\yoc{b\restr n}$ tend to zero, since the initial segments $b\restr n$ are distinct.

For $T\in\Tr$ and $s\in\Nseq$, set
\begin{equation}\label{eq:tree-penalties}
 \yoa{T}{s}=\max\left(\{\yoc{s}\}\cup
  \set*{\yoc{u}\middle|u\preccurlyeq s,\ u\notin T}\right).
\end{equation}
Thus $\yoa{T}{s}$ is the largest of the weight $\yoc{s}$ and the weights of all initial segments of $s$ missing from $T$. If $s\in T$, all its initial segments belong to $T$, so $\yoa{T}{s}=\yoc{s}$. If $u\notin T$, every extension $s$ of $u$ satisfies $\yoa{T}{s}\geq\yoc{u}>0$. In the metric defined below, $\yoa{T}{s}$ is the distance from the vertex $\yov{s}$ to $P$. Vertices extending a missing initial segment $u$ therefore stay at distance at least $\yoc{u}$ from $P$, whereas along an infinite branch $b$ of $T$ these distances are $\yoc{b\restr n}$ and tend to zero.

For $T\in\Tr$, define a symmetric function $\yod{T}$ on $\TreeBoundarySpace^2$ by
\begin{equation}\label{eq:tree-metric}
 \yod{T}(a,b)=
 \begin{cases}
  \yotreerho(a,b),
    & a,b\in P,\\
  \yoa{T}{s}+\yotreerho(\yotreeext{s},x),
    & (a,b)=(\yov{s},x),\ s\in\Nseq,\ x\in P,\\
  \yoa{T}{s}+\yotreerho(\yotreeext{s},\yotreeext{t})+\yoa{T}{t},
    & (a,b)=(\yov{s},\yov{t}),\ s,t\in\Nseq,\ s\ne t,\\
  0,
    & (a,b)=(\yov{s},\yov{s}),\ s\in\Nseq.
 \end{cases}
\end{equation}
The mixed case also determines $\yod{T}(x,\yov{s})$ by symmetry. In particular, $\inf_{x\in P}\yod{T}(\yov{s},x)=\yoa{T}{s}$, with the minimum attained at $x=\yotreeext{s}$. We will prove that $\yod{T}$ is always a bounded continuous metric and determine when it generates the fixed topology of $\TreeBoundarySpace$.

\subsection{Continuity and compatibility}

We study the map $T\mapsto\yod{T}$ defined by \eqref{eq:tree-metric} on $\TreeBoundarySpace$. It is continuous for the supremum distance, and compatibility of $\yod{T}$ is equivalent to well-foundedness of $T$.

\begin{theorem}\label{thm:tree-reduction}
The map
\[
 \Phi\colon\Tr\to\CPM(\TreeBoundarySpace),\qquad \Phi(T)=\yod{T}
\]
is continuous.
For every
$T\in\Tr$,
the function
$\yod{T}$
is a bounded continuous metric satisfying
\begin{equation}\label{eq:tree-compatibility}
 \yod{T}\in\Met(\TreeBoundarySpace)\quad\Longleftrightarrow\quad T\in\WF.
\end{equation}
Consequently,
$\Met(\TreeBoundarySpace)$
is {non-Borel} in
$\CPM(\TreeBoundarySpace)$.
\end{theorem}

\begin{proof}
First, \eqref{eq:tree-metric} defines a bounded metric. The triangle inequalities follow from those for $\yotreerho$ and the nonnegative terms $\yoa{T}{s}$. Each restriction to a product of the open summands $P$ and the singletons $\set{\yov{s}}$ is continuous, so $\yod{T}\in\CPM(\TreeBoundarySpace)$. To prove continuity of $\Phi$, take $T,U\in\Tr$. If exactly one of $T,U$ contains $u\in\Nseq$, then \eqref{eq:tree-parameter-metric} and \eqref{eq:tree-weights} imply $\yoc{u}\leq\Delta(T,U)$. For each $s\in\Nseq$, the terms in the two maxima defining $\yoa{T}{s}$ and $\yoa{U}{s}$ are therefore either common to both or at most $\Delta(T,U)$. Hence
\begin{equation}\label{eq:tree-penalty-continuity}
 \sup_{s\in\Nseq}\abs{\yoa{T}{s}-\yoa{U}{s}}\leq\Delta(T,U).
\end{equation}
Substitution in \eqref{eq:tree-metric} yields
\begin{equation}\label{eq:tree-continuity}
 \yometdis{\TreeBoundarySpace}(\yod{T},\yod{U})\leq2\Delta(T,U).
\end{equation}
Thus $\Phi$ is $2$-Lipschitz from $(\Tr,\Delta)$ to $(\CPM(\TreeBoundarySpace),\yometdis{\TreeBoundarySpace})$, and in particular is continuous.

Fix $T\in\Tr$. For $b\in P$, define its distance from $V$ by
\begin{equation}\label{eq:tree-cross-distance}
 \eta_T(b)=\inf_{s\in\Nseq}
 \bigl(\yoa{T}{s}+\yotreerho(\yotreeext{s},b)\bigr).
\end{equation}
The metric $\yod{T}$ induces the fixed topology on $P$ and isolates every point of $V$, since $\yod{T}(\yov{s},a)\geq\yoa{T}{s}>0$ for $a\ne\yov{s}$. It therefore generates the topology of $\TreeBoundarySpace$ exactly when $\eta_T(b)>0$ for every $b\in P$.

If $b$ is an infinite branch of $T$, its initial segments $s_n=b\restr n$ belong to $T$. Thus
\[
 \yoa{T}{s_n}+\yotreerho(\yotreeext{s_n},b)
 \leq\yoc{s_n}+2^{-n}\to0
\]
by \eqref{eq:baire-tree-distance} and \eqref{eq:tree-penalties}. Consequently $\eta_T(b)=0$, so $\yod{T}$ is not compatible.

Conversely, assume $\eta_T(b)=0$ for some $b\in P$. {Choose a sequence $\{s_j\}_{j\in\yoomega}$ in $\Nseq$ such that, as $j\to\infty$,}
\begin{equation}\label{eq:tree-cross-vanishing}
 \begin{aligned}
 \yoa{T}{s_j}&\to0,\\
 \yotreerho(\yotreeext{s_j},b)&\to0.
 \end{aligned}
\end{equation}
Fix $k\in\yoomega$. Eventually $\yotreeext{s_j}\restr k=b\restr k$. Among nodes shorter than $k$, only proper initial segments of $b\restr k$ can have a zero extension with this property. There are at most $k$ such nodes. Their positive weights, together with $\yoa{T}{s_j}\geq\yoc{s_j}$ and \eqref{eq:tree-cross-vanishing}, imply that eventually $|s_j|\geq k$. Hence $b\restr k\preccurlyeq s_j$ eventually. If $b\restr k$ were absent from $T$, \eqref{eq:tree-penalties} would imply
\[
 \yoa{T}{s_j}\geq\yoc{b\restr k}>0
\]
eventually, contradicting \eqref{eq:tree-cross-vanishing}. Thus every $b\restr k$ belongs to $T$, and $b$ is an infinite branch.

We have proved \eqref{eq:tree-compatibility}. In particular, $\Phi^{-1}(\Met(\TreeBoundarySpace))=\WF$. Since $\Phi$ is continuous and $\WF$ is non-Borel by \Cref{thm:wf-nonborel}, $\Met(\TreeBoundarySpace)$ is non-Borel in $\CPM(\TreeBoundarySpace)$.
\end{proof}

A closed copy of the zero-dimensional Baire space contains a closed copy of $\TreeBoundarySpace$, so the preceding family extends to any space containing such a copy.

\begin{corollary}\label{cor:closed-baire}
If a metrizable space
$X$
contains a closed subspace homeomorphic to
$\Baire$,
then
$\Met(X)$
is {non-Borel} in
$\CPM(X)$.
\end{corollary}

\begin{proof}
Consider the clopen cylinder
\[
 C_0=\set{x\in\Baire\mid x(0)=0}
\]
and the points $p_n=(n+1,0,0,\ldots)$ for $n\in\yoomega$. The subspace
\[
 C_0\cup\set{p_n\mid n\in\yoomega}
\]
is closed in $\Baire$ and homeomorphic to $\TreeBoundarySpace=P\sqcup V$. Hence there are a closed $B\subseteq X$ and a homeomorphism $\iota\colon \TreeBoundarySpace\to B$. Transport the tree family by setting, for $T\in\Tr$ and $x,y\in B$,
\begin{equation}\label{eq:closed-baire-transport}
 q_T(x,y)=\yod{T}(\iota^{-1}(x),\iota^{-1}(y)).
\end{equation}
For $T,U\in\Tr$,
\begin{equation}\label{eq:closed-baire-transport-isometry}
 \yometdis{B}(q_T,q_U)=\yometdis{\TreeBoundarySpace}(\yod{T},\yod{U}),
\end{equation}
so $T\mapsto q_T$ is continuous. The homeomorphism $\iota$ also preserves compatibility. Choose the extensor $E\colon\CPM(B)\to\CPM(X)$ of \Cref{prop:extensor}. Then
\[
 F\colon\Tr\to\CPM(X),\qquad F(T)=E(q_T)
\]
is continuous and satisfies
\[
 F^{-1}(\Met(X))=\WF.
\]
The non-Borelness of $\WF$ now implies that of $\Met(X)$.
\end{proof}

\section{Metrics defined using a compact completion}\label{sec:boundary}
We treat a metrizable space $X$ that is non-Borel in a compact metrizable space containing it densely. This case, together with \Cref{cor:closed-baire,thm:hurewicz}, will prove in \Cref{thm:borel-sigma-compact} that Borelness of $\Met(X)$ in $\CPM(X)$ forces $\sigma$-compactness of $X$. We employ the idea of Koshino's proof of \cite[Theorem~1.1]{koshino-borel}.

\begin{proposition}\label{prop:nonborel-completion}
Let
$X$
be a separable metrizable space,
and let
$Y$
be a compact metrizable
space in which
$X$
is dense.
If
$X$
is {non-Borel} in
$Y$,
then
$\Met(X)$
is
{non-Borel} in
$\CPM(X)$.
More precisely,
there exist a compact metrizable space
$K$,
a non-Borel set
$S\subseteq K$,
and a continuous map
$F\colon K\to\CPM(X)$
such that
$F(z)$
is bounded for every
$z\in K$
and
\[
 F^{-1}(\Met(X))=S.
\]
\end{proposition}

\begin{proof}
Choose a compatible metric $\rho$ on $Y$, and set
\[
 \begin{aligned}
 I&=\Int_Y(X), & B&=X\setminus I,\\
 C&=\cl_Y(B), & Z&=C\setminus X.
 \end{aligned}
\]
Since $I$ is open and $X=B\cup I$ is non-Borel in $Y$, the set $B$ is non-Borel in $Y$. Also, $B$ is closed relative to $X$, so
\begin{equation}\label{eq:boundary-closure-intersection}
 C\cap X=B.
\end{equation}
Consequently $Z=C\setminus B$ is non-Borel in the compact Polish space $C$.

Fix $a\in Z$. For $n\in\yoomega$, set
\[
 K_n=\set{z\in C\mid \rho(z,a)\geq2^{-n}}.
\]
The decomposition
\[
 Z=\set{a}\cup\bigcup_{n\in\yoomega}(Z\cap K_n)
\]
implies that $Z\cap K_n$ is non-Borel in $K_n$ for at least one $n$. Fix such an $n$ and put
\[
 K=K_n,\qquad Z_K=Z\cap K.
\]

By density of $X$ in $Y$, {choose a sequence $\{a_m\}_{m\in\yoomega}$ of distinct points of $X$ tending to $a$, with} $\rho(a_m,a)<2^{-(n+1)}$ for every $m$. Put
\[
 A_0=\set{a_m\mid m\in\yoomega},\qquad B_0=B\cap K.
\]
The set $A_0$ is closed in $X$, since its only accumulation point in $Y$ is $a\notin X$. The set $B_0$ is also closed in $X$. For $x\in A_0$ and $y\in B_0$,
\[
 \rho(x,y)\geq\rho(y,a)-\rho(x,a)>2^{-(n+1)},
\]
so $A_0\cup B_0$ is the topological sum of these two closed subspaces.

Fix $z\in K$. Define a symmetric function $e_z\colon(A_0\cup B_0)^2\to[0,\infty)$ as follows. On $A_0^2$ and $B_0^2$, set $e_z=\rho$, and, for $x\in A_0$ and $y\in B_0$, set
\begin{equation}\label{eq:boundary-cross-distance}
 e_z(x,y)=\rho(x,a)+\rho(y,z).
\end{equation}
Positivity follows from $a\notin A_0$. For $x_0,x_1\in A_0$, the inequality
\[
 \rho(x_0,x_1)\leq\rho(x_0,a)+\rho(a,x_1)
\]
verifies the triangle inequality with the side inside $A_0$ on the left. The corresponding inequality through $z$ treats the side inside $B_0$. The other mixed inequalities are, for $x_0,x_1\in A_0$ and $y\in B_0$,
\[
 e_z(x_0,y)\leq e_z(x_0,x_1)+e_z(x_1,y),
\]
and, for $x\in A_0$ and $y_0,y_1\in B_0$,
\[
 e_z(x,y_0)\leq e_z(x,y_1)+e_z(y_1,y_0).
\]
These follow directly from the triangle inequality for $\rho$. Thus $e_z$ is a metric. The defining formulas are continuous on each product of summands, so $e_z\in\CPM(A_0\cup B_0)$.

Define the parameter map by
\[
 q\colon K\to\CPM(A_0\cup B_0),\qquad q(z)=e_z.
\]
Let $z,w\in K$. On $A_0^2$ and $B_0^2$, both $e_z$ and $e_w$ equal $\rho$. For $x\in A_0$ and $y\in B_0$,
\[
 \abs{e_z(x,y)-e_w(x,y)}
 =\abs{\rho(y,z)-\rho(y,w)}\leq\rho(z,w).
\]
Hence
\begin{equation}\label{eq:boundary-parameter-continuity}
 \yometdis{A_0\cup B_0}(e_z,e_w)\leq\rho(z,w)
\end{equation}
for all $z,w\in K$, and $q$ is $1$-Lipschitz. We claim that
\begin{equation}\label{eq:boundary-compatibility}
 e_z\in\Met(A_0\cup B_0)\quad\Longleftrightarrow\quad z\in Z_K
\end{equation}
for every $z\in K$. By \eqref{eq:boundary-closure-intersection}, $K\cap X=B_0$. If $z\in B_0$, then
\[
 e_z(a_m,z)=\rho(a_m,a)\to0,
\]
whereas the sequence $\{a_m\}_{m\in\yoomega}$ cannot converge to a point of $B_0$ in the fixed topological sum. Thus $e_z$ is not compatible.

If $z\in Z_K=K\setminus X$, then each $x\in A_0$ has distance at least $\rho(x,a)>0$ from $B_0$, and each $y\in B_0$ has distance at least $\rho(y,z)>0$ from $A_0$. Every point therefore has an $e_z$-ball contained in its original summand. Since $e_z$ restricts to $\rho$ on both summands, it generates the fixed topology. This proves \eqref{eq:boundary-compatibility}.

\Cref{fig:completion-parameters} illustrates these two cases.

\begin{figure}[tbp]
\centering
\begin{tikzpicture}[x=1cm,y=1cm,font=\small,
  outline/.style={line width=.65pt},
  point/.style={circle,fill=black,inner sep=1.05pt}]
  \begin{scope}
    \node at (2.15,4.35) {$z\in B_0$};
    \draw[outline] (1.35,2.6)--(.55,2.5)--(-.05,1.35)
      --(.55,.05)--(3.95,.05)--(4.55,1.25)--(4,2.35)--(2.85,2.6);
    \node[anchor=west] at (4.15,.4) {$K$};
    \draw[outline,fill=black!3] (2.2,1.25) ellipse (1.45 and .86);
    \node[anchor=west] at (3.6,1.15) {$B_0$};
    \draw[outline,fill=white] (2.05,1.45)
      .. controls (1.7,2.35) and (1.75,3.6) .. (2.1,3.95)
      .. controls (2.58,3.55) and (2.4,2.3) .. cycle;
    \node[anchor=west] at (2.4,3.6) {$A_0$};
    \foreach \x/\y in {2.12/3.48,2.12/2.95,2.10/2.48,2.09/2.12,2.07/1.86,2.06/1.69}{
      \node[point] at (\x,\y) {};
    }
    \node[point,inner sep=1.7pt,label=below:{$z$}] at (2.05,1.45) {};
  \end{scope}
  \begin{scope}[xshift=6cm]
    \node at (2.15,4.35) {$z\in K\setminus X$};
    \draw[outline] (1.05,2.6)--(.55,2.5)--(-.05,1.35)
      --(.55,.05)--(3.95,.05)--(4.55,1.25)--(4,2.35)--(2.1,2.6);
    \node[anchor=west] at (4.15,.4) {$K$};
    \draw[outline,fill=black!3] (2.6,1.25) ellipse (1.15 and .76);
    \node[anchor=west] at (3.78,1.15) {$B_0$};
    \draw[outline,fill=white] (.95,1.45)
      .. controls (1.08,2.3) and (1.72,3.5) .. (2.08,3.72)
      .. controls (2.34,3.63) and (1.4,2.1) .. cycle;
    \node[anchor=west] at (2.23,3.65) {$A_0$};
    \foreach \x/\y in {1.94/3.35,1.71/2.93,1.46/2.49,1.27/2.14,1.12/1.86,1.04/1.66}{
      \node[point] at (\x,\y) {};
    }
    \draw[outline,fill=white] (.95,1.45) circle (.055);
    \node[below] at (.95,1.39) {$z$};
  \end{scope}
\end{tikzpicture}
\caption{Two configurations for the parameter $z\in K$. The oval denotes $B_0=K\cap X$. The narrowing part depicts the discrete sequence $A_0$ approaching $z$ in the space obtained by extending \eqref{eq:boundary-cross-distance} to $y\in K$. On the left, $z\in B_0$, so the sequence converges to a point of $A_0\cup B_0$ for $e_z$. On the right, $z\in K\setminus X$, so its limit lies outside $A_0\cup B_0$.}
\label{fig:completion-parameters}
\end{figure}

Apply \Cref{prop:extensor} to the closed subspace $A_0\cup B_0$ of $X$, and let $E\colon\CPM(A_0\cup B_0)\to\CPM(X)$ be the resulting extensor. The continuous map
\[
 F=E\circ q\colon K\to\CPM(X)
\]
satisfies
\begin{equation}\label{eq:boundary-reduction}
 F^{-1}(\Met(X))=Z_K
\end{equation}
by \eqref{eq:extensor-reduction} and \eqref{eq:boundary-compatibility}. Since $Z_K$ is non-Borel in $K$, $\Met(X)$ is non-Borel in $\CPM(X)$. Each $e_z$ is bounded by $2\diam_\rho(Y)$, and $E$ preserves boundedness by \Cref{prop:extensor}. Thus $F(z)$ is bounded for every $z\in K$, and $S=Z_K$ has the asserted properties.
\end{proof}

\section{The characterization theorem}\label{sec:main}
We assume that $\Met(X)$ is Borel in $\CPM(X)$ and prove that $X$ is $\sigma$-compact. By \Cref{thm:borel-separable}, $X$ is separable, so \Cref{prop:compact-completion} embeds it densely in a compact metrizable space $Y$. The cases below concern whether the subset $X$ of $Y$ is Borel in $Y$.

\begin{theorem}\label{thm:borel-sigma-compact}
Let
$X$
be a metrizable space.
If
$\Met(X)$
is Borel in
$\CPM(X)$,
then
$X$
is
$\sigma$-compact.
\end{theorem}

\begin{proof}
The separability theorem \Cref{thm:borel-separable} applies to $X$. By \Cref{prop:compact-completion}, choose a compact metrizable space $Y$ containing $X$ as a dense subspace.

If $X$ were non-Borel in $Y$, \Cref{prop:nonborel-completion} would make $\Met(X)$ non-Borel in $\CPM(X)$, contrary to the hypothesis. Thus $X$ is Borel in $Y$.

If $X$ were not $\sigma$-compact, \Cref{thm:hurewicz} would provide a closed subspace of $X$ homeomorphic to $\Baire$. By \Cref{cor:closed-baire}, $\Met(X)$ would again be non-Borel. This contradiction proves $\sigma$-compactness of $X$.
\end{proof}

We complete the proof of \Cref{thm:introduction}, which asserts the equivalence of Borelness of $\Met(X)$ in $\CPM(X)$, $\sigma$-compactness of $X$, and complete metrizability of $\Met(X)$. The remaining implications follow from \Cref{thm:koshino,thm:complete-subspaces-gdelta}.

\begin{proof}[Proof of \Cref{thm:introduction}]
By \Cref{thm:borel-sigma-compact}, \ref{item:intro-met-borel} implies \ref{item:intro-met-sigma}. A $\sigma$-compact metrizable space is separable, so \Cref{thm:koshino} proves \ref{item:intro-met-sigma} $\Rightarrow$ \ref{item:intro-met-complete}. Finally, if $\Met(X)$ is completely metrizable, \Cref{thm:complete-subspaces-gdelta} makes it $G_\delta$ in the metrizable space $\CPM(X)$ and hence Borel. This proves \ref{item:intro-met-complete} $\Rightarrow$ \ref{item:intro-met-borel}.
\end{proof}

\section{Components of the space of metrics}\label{sec:components}
We prove \Cref{thm:introduction-components}. When $X$ is $\sigma$-compact, the result follows from complete metrizability of $\Met(X)$. Otherwise, we place a family witnessing non-Borelness {in the component $\yopclq{d}$ of $\CPM(X)$ for each prescribed $d\in\Met(X)$.}

\subsection{Components and bounded metrics}

For a compatible metric $d$ on $X$, we first describe {the component $\yopclq{d}$ of $d$ in $\CPM(X)$} and its subspace $\yocclq{d}$ of compatible metrics. We then establish the bounded version of non-Borelness used in the proof of \Cref{thm:introduction-components}.

\begin{proposition}[{\cite{ishiki-extensor}}]\label{lem:metric-components}
Let
$X$
be metrizable and let
$d\in\Met(X)$.
Then the space
$\yopclq{d}$
is completely metrizable,
and
$\yocclq{d}$
is clopen in
$\Met(X)$.
Moreover,
$\yocclq{d}$
is both the
connected component and the path component of
$d$
in
$\Met(X)$.
\end{proposition}



{For a metrizable space $Y$, let $\yocpmbd{Y}$ and $\yometbd{Y}$ be the subspaces consisting of bounded elements of $\CPM(Y)$ and $\Met(Y)$, respectively. We first show that $\yometbd{Y}$ is non-Borel in $\yocpmbd{Y}$ when $Y$ is not $\sigma$-compact.}

\begin{proposition}\label{lem:bounded-nonborel}
If a metrizable space
$Y$
is not
$\sigma$-compact,
then
$\yometbd{Y}$
is non-Borel in
$\yocpmbd{Y}$.
\end{proposition}

\begin{proof}
Assume first that $Y$ is separable. By \Cref{prop:compact-completion}, embed it densely in a compact metrizable space. If $Y$ is non-Borel there, use the bounded family of \Cref{prop:nonborel-completion}. If $Y$ is Borel, \Cref{thm:hurewicz} and non-$\sigma$-compactness provide a closed copy of $\Baire$. In that case, use the bounded tree family of \Cref{thm:tree-reduction} through \Cref{cor:closed-baire}.

In either case, the initial family consists of bounded continuous pseudometrics, and \Cref{prop:extensor} preserves boundedness. Thus the resulting family takes values in $\yocpmbd{Y}$, the component of the zero pseudometric in $\CPM(Y)$.

The inverse image of $\yometbd{Y}$ under this continuous map is non-Borel. Therefore $\yometbd{Y}$ is non-Borel in $\yocpmbd{Y}$.

If $Y$ is nonseparable, choose a closed discrete $Z\subseteq Y$ of cardinality $\aleph_1$. The space $\yometbd{Z}$ is non-Borel in $\yocpmbd{Z}$ by \cite[Theorem~3.3]{ishiki-uda-borel}. Choose an isometric extensor
\[
 E\colon\CPM(Z)\to\CPM(Y)
\]
as in \Cref{prop:extensor}. It preserves boundedness, and \eqref{eq:extensor-compatibility} implies
\[
 E^{-1}\bigl(\yometbd{Y}\bigr)
 =\yometbd{Z}
 \qquad\text{on }\yocpmbd{Z}.
\]
Since $E$ is continuous, $\yometbd{Y}$ is non-Borel.
\end{proof}

\subsection{Characterization of components}

To prove \Cref{thm:introduction-components}, the main step is the case of separable, non-$\sigma$-compact $X$. For a fixed $d\in\Met(X)$, choose a non-$\sigma$-compact closed $d$-bounded ball $B$. We construct a continuous map from the bounded continuous pseudometrics on $B$ to $\yopclq{d}$, preserving compatibility in both directions. By \Cref{lem:bounded-nonborel}, this proves non-Borelness of $\yocclq{d}$.

\begin{proof}[Proof of \Cref{thm:introduction-components}]
Fix $d\in\Met(X)$. By \Cref{lem:metric-components}, $\yocclq{d}$ is the connected and path component of $d$ and is clopen in $\Met(X)$. Since $\yopclq{d}$ is clopen in $\CPM(X)$, the set $\yocclq{d}$ is Borel in $\CPM(X)$ if and only if it is Borel in $\yopclq{d}$.

We first prove \ref{item:intro-metric-component-sigma} $\Rightarrow$ \ref{item:intro-metric-component-complete}. If $X$ is $\sigma$-compact, then $\Met(X)$ is completely metrizable by \Cref{thm:introduction}. Its closed subspace $\yocclq{d}$ is therefore completely metrizable.

For \ref{item:intro-metric-component-complete} $\Rightarrow$ \ref{item:intro-metric-component-borel}, assume that $\yocclq{d}$ is completely metrizable. By \Cref{thm:complete-subspaces-gdelta}, it is $G_\delta$ in $\yopclq{d}$ and hence Borel in $\CPM(X)$.

It remains to prove \ref{item:intro-metric-component-borel} $\Rightarrow$ \ref{item:intro-metric-component-sigma}. Assume that $X$ is not $\sigma$-compact. We prove that $\yocclq{d}$ is non-Borel in $\CPM(X)$. If $X$ is nonseparable, this follows from \cite[Theorem~3.7(M1)]{ishiki-uda-borel}. We may therefore assume $X$ separable. Fix $x_0\in X$ and set, for $n\in\yoomega$,
\[
 B_n=\set*{x\in X\middle|d(x,x_0)\leq n}.
\]
These balls cover $X$, so at least one is not $\sigma$-compact. Fix such a ball and denote it by $B$. It is closed in $X$ and bounded for $d$.

{For $x\in X$, write $a(x)=d(x,B)$. For $x,y\in X$, define}
\begin{equation}\label{eq:collapsed-pseudometric}
 r_B(x,y)=\min\set{d(x,y),a(x)+a(y)}.
\end{equation}
{This is the quotient pseudometric} obtained by identifying the points of $B$. Equivalently, distances are computed by allowing zero cost between any two points of $B$. This description proves the triangle inequality. The formula also shows that $r_B$ vanishes on $B^2$ and is continuous, since $a$ is $1$-Lipschitz. We have
\begin{equation}\label{eq:collapsed-bound}
 \yometdis{X}(r_B,d)\leq\diam_d(B)<\infty,
\end{equation}
{because, for all $x,y\in X$, approximation of their distances to $B$ yields}
\[
 d(x,y)\leq a(x)+\diam_d(B)+a(y).
\]
{}

Choose the extensor $E\colon\CPM(B)\to\CPM(X)$ of \Cref{prop:extensor}. For $q\in\yocpmbd{B}$, define
\begin{equation}\label{eq:metric-component-transfer}
 F(q)=r_B+E(q).
\end{equation}
For $q_0,q_1\in\yocpmbd{B}$, the isometry of $E$ yields
\begin{equation}\label{eq:metric-component-isometry}
 \begin{aligned}
 \yometdis{X}(F(q_0),F(q_1))
 &=\sup_{x,y\in X}\Bigl|
   \bigl(r_B+E(q_0)\bigr)(x,y)
   -\bigl(r_B+E(q_1)\bigr)(x,y)\Bigr|\\
 &=\sup_{x,y\in X}\abs{E(q_0)(x,y)-E(q_1)(x,y)}\\
 &=\yometdis{X}(E(q_0),E(q_1))\\
 &=\yometdis{B}(q_0,q_1)
 \end{aligned}.
\end{equation}
Set $e=d\restr B^2$. Since $E(e)$ is bounded, \eqref{eq:collapsed-bound} implies
\begin{equation}\label{eq:metric-component-anchor}
 \yometdis{X}(F(e),d)
 \leq\diam_d(B)+\sup_{x,y\in X}E(e)(x,y)<\infty.
\end{equation}
Every bounded $q$ has finite supremum distance from $e$. Thus \eqref{eq:metric-component-isometry} and \eqref{eq:metric-component-anchor} imply $F(q)\in\yopclq{d}$. We have obtained a continuous map $F\colon\yocpmbd{B}\to\yopclq{d}$. {For every $q\in\yocpmbd{B}$, the vanishing of $r_B$ on $B^2$ gives}
\begin{equation}\label{eq:metric-component-restriction}
 F(q)\restr B^2=q.
\end{equation}
{}

If $q\in\yometbd{B}$, the metric $E(q)$ is compatible. The sum $F(q)=r_B+E(q)$ is continuous and dominates $E(q)$, so it is compatible as well. Conversely, if $F(q)$ is compatible, its restriction \eqref{eq:metric-component-restriction} makes $q$ compatible on $B$. Therefore
\begin{equation}\label{eq:metric-component-reduction}
 F^{-1}(\yocclq{d})=\yometbd{B}.
\end{equation}
By \Cref{lem:bounded-nonborel}, the right side is non-Borel. Continuity of $F$ implies that $\yocclq{d}$ is non-Borel. We have shown that failure of $\sigma$-compactness implies failure of \ref{item:intro-metric-component-borel}, which completes the proof.
\end{proof}

\section{Complete compatible metrics}\label{sec:complete-metrics}
To study complete compatible metrics, we construct metrics on a countable discrete space whose completeness is equivalent to well-foundedness of a tree. We allow any strictly decreasing null sequence of positive real numbers, so that the same criterion can later be used for prescribed range ultrametrics.

\subsection{A continuous family on a countable discrete space}

Equip $\Nseq$ with the discrete topology. For each $T\in\Tr$, we construct a compatible metric $\yocmpd{T}$ on $\Nseq$. For fixed $T$ and $s\in\Nseq$, the distances $\yocmpd{T}(s,t)$ for $t\in\Nseq\setminus\{s\}$ have a positive lower bound, so $s$ is isolated. If $b\in\Baire$ is an infinite branch of $T$, the sequence $\{b\restr n\}_{n\in\yoomega}$ is nevertheless Cauchy in $(\Nseq,\yocmpd{T})$.

\begin{proposition}\label{lem:discrete-tree-completeness}
Equip
$\Nseq$
with the discrete topology,
and let
$\{r(n)\}_{n\in\yoomega}$
be a strictly decreasing null sequence of positive real numbers.
Let
$\{\yocmpc{s}\}_{s\in \Nseq}$
be a family in
$(0,r(0)]$
such that,
for every
$\epsilon>0$,
the set
$\set{s\in \Nseq\mid\yocmpc{s}\geq\epsilon}$
is finite.
For
$s\in \Nseq$,
let
$\yocmpext{s}\in\Baire$
be its zero extension as defined in \eqref{eq:zero-extension}.
For
$x,y\in\Baire$,
put
\begin{equation}\label{eq:discrete-tree-background}
 \yocmprho(x,y)=
 \begin{cases}
  0,&x=y,\\
  r(\min\{n\in\yoomega\mid x(n)\ne y(n)\}),&x\ne y.
 \end{cases}
\end{equation}
For
$T\in\Tr$
and
$s\in \Nseq$,
define
\begin{equation}\label{eq:discrete-tree-penalties}
 \yocmpa{T}{s}=\max\left(\{\yocmpc{s}\}\cup
 \set*{\yocmpc{u}\middle|u\preccurlyeq s,\ u\notin T}\right).
\end{equation}
For
$T\in\Tr$
and
$s,t\in \Nseq$,
define
\begin{equation}\label{eq:discrete-tree-distance}
 \yocmpd{T}(s,t)=
 \begin{cases}
  0,&s=t,\\
  \yocmpa{T}{s}\lor\yocmprho(\yocmpext{s},\yocmpext{t})\lor\yocmpa{T}{t},
     &s\ne t.
 \end{cases}
\end{equation}
Then,
for every
$T\in\Tr$,
$\yocmpd{T}$
is a compatible ultrametric on
$\Nseq$
bounded by
$r(0)$,
and
\begin{equation}\label{eq:discrete-tree-completeness}
 \yocmpd{T}\in\Comp(\Nseq)\quad\Longleftrightarrow\quad T\in\WF.
\end{equation}
\end{proposition}

\begin{proof}
The balls of $\yocmprho$ are cylinders, and a Cauchy sequence has an eventually constant initial segment of each length. Thus $\yocmprho$ is compatible and complete on $\Baire$. Fix $T\in\Tr$. For distinct $s,t,u\in \Nseq$, the strong triangle inequality for $\yocmprho$ yields
\[
 \yocmpd{T}(s,t)\leq\yocmpd{T}(s,u)\lor\yocmpd{T}(u,t)
\]
by \eqref{eq:discrete-tree-distance}. The same inequality is immediate when nodes repeat. For $s\ne t$,
\[
 0<\yocmpa{T}{s}\leq\yocmpd{T}(s,t)\leq r(0),
\]
so $\yocmpd{T}$ is an ultrametric bounded by $r(0)$, and its open ball at $s$ of radius $\yocmpa{T}{s}$ is the singleton $\set{s}$. Hence it generates the discrete topology.

Let $b$ be an infinite branch of $T$ and set $s_n=b\restr n$. Since $s_n\in T$, we have $\yocmpa{T}{s_n}=\yocmpc{s_n}$. The nodes $s_n$ are distinct, so their weights tend to zero. Moreover,
\begin{equation}\label{eq:discrete-tree-branch-cauchy}
 \yocmpd{T}(s_n,s_m)
 \leq \yocmpc{s_n}\lor r(\min\{m,n\})\lor\yocmpc{s_m}
 \to0
 \quad(m,n\to\infty).
\end{equation}
Thus $\{s_n\}_{n\in\yoomega}$ is Cauchy. It cannot converge in the discrete space $\Nseq$, since it is not eventually constant. Therefore $\yocmpd{T}$ is incomplete.

Conversely, let $\{s_j\}_{j\in\yoomega}$ be a Cauchy sequence for $\yocmpd{T}$ that is not eventually constant. Since every node is isolated, no node can occur infinitely often. We may pass to a subsequence of distinct nodes. {Formula \eqref{eq:discrete-tree-distance} gives the following limits, respectively as $j\to\infty$ and as $j,k\to\infty$:}
\begin{equation}\label{eq:discrete-tree-cauchy-limits}
 \yocmpa{T}{s_j}\to0
\end{equation}
and
\begin{equation}\label{eq:discrete-tree-extension-cauchy}
 \yocmprho(\yocmpext{s_j},\yocmpext{s_k})\to0.
\end{equation}
Completeness of $(\Baire,\yocmprho)$ yields $b\in\Baire$ with $\yocmpext{s_j}\to b$. Fix $k\in\yoomega$. Eventually $\yocmpext{s_j}\restr k=b\restr k$. The nodes shorter than $k$ with this property are among the finitely many proper initial segments of $b\restr k$. Their positive weights and \eqref{eq:discrete-tree-cauchy-limits} exclude them for large $j$. Hence $b\restr k\preccurlyeq s_j$ eventually. If $b\restr k\notin T$, then \eqref{eq:discrete-tree-penalties} implies
\[
 \yocmpa{T}{s_j}\geq\yocmpc{b\restr k}>0
\]
for large $j$, contradicting \eqref{eq:discrete-tree-cauchy-limits}. Thus $b\restr k\in T$ for every $k$, and $b$ is an infinite branch.

Thus a nonconvergent Cauchy sequence exists exactly when $T$ has an infinite branch, proving \eqref{eq:discrete-tree-completeness}.
\end{proof}

We choose the sequence $\{r(n)\}_{n\in\yoomega}$ and the weights in \Cref{lem:discrete-tree-completeness} so that $T\mapsto\yocmpd{T}$ is continuous from $\Tr$ to $\Met(\Nseq)$ for the supremum distance. Using the bijection $\yotreeNu$ fixed in \Cref{def:tree-space}, set
\[
 r(n)=2^{-n-1}
\]
{for $n\in\yoomega$, and, for $s\in\Nseq$, define the weight}
\[
 \yocmpc{s}=2^{-\yotreeNu(s)-1}.
\]
{These choices satisfy} the hypotheses of \Cref{lem:discrete-tree-completeness}. Denote the resulting metric \eqref{eq:discrete-tree-distance} by $\yocmpd{T}$.

\begin{theorem}\label{thm:complete-tree-reduction}
The map
\[
 \Psi\colon\Tr\to\Met(\Nseq),\qquad \Psi(T)=\yocmpd{T}
\]
is continuous for the supremum topology,
and for every
$T\in\Tr$,
$\yocmpd{T}$
is a bounded compatible ultrametric satisfying
\begin{equation}\label{eq:complete-tree-equivalence}
 \yocmpd{T}\in\Comp(\Nseq)\quad\Longleftrightarrow\quad T\in\WF.
\end{equation}
Consequently,
$\Comp(\Nseq)$
is {non-Borel} even relative to
$\Met(\Nseq)$.
\end{theorem}

\begin{proof}
Take $T,U\in\Tr$. If $u\in\Nseq$ belongs to exactly one of $T,U$, then $\yocmpc{u}=2^{-\yotreeNu(u)-1}\leq\Delta(T,U)/2$. Thus each term in the maxima defining $\yocmpa{T}{s}$ and $\yocmpa{U}{s}$ is either common to both or at most $\Delta(T,U)/2$. Consequently,
\begin{equation}\label{eq:complete-tree-penalty-continuity}
 \sup_{s\in \Nseq}\abs{\yocmpa{T}{s}-\yocmpa{U}{s}}
 \leq\frac12\Delta(T,U).
\end{equation}
Taking a maximum is $1$-Lipschitz for the supremum norm. Since the background ultrametric $\yocmprho$ is independent of $T$, \eqref{eq:discrete-tree-distance} yields
\begin{equation}\label{eq:complete-tree-continuity}
 \yometdis{\Nseq}(\yocmpd{T},\yocmpd{U})
 \leq\frac12\Delta(T,U).
\end{equation}
Hence $\Psi$ is $1/2$-Lipschitz on $(\Tr,\Delta)$ and is therefore continuous.

By \Cref{lem:discrete-tree-completeness}, the inverse image of $\Comp(\Nseq)$ under $\Psi$ is $\WF$. This proves \eqref{eq:complete-tree-equivalence} and, since $\WF$ is non-Borel by \Cref{thm:wf-nonborel}, the asserted relative non-Borelness.
\end{proof}

\subsection{The characterization of complete metrics}

We prove \Cref{thm:introduction-comp}. For noncompact completely metrizable $X$, transfer the tree family of \Cref{thm:complete-tree-reduction} from the discrete space $\Nseq$ to a countably infinite closed discrete subspace of $X$. The completeness-preserving extensor of \Cref{prop:extensor} then gives a continuous family in $\Met(X)$ whose complete members have exactly the well-founded trees as parameters.

\begin{proof}[Proof of \Cref{thm:introduction-comp}]
Assume $X$ completely metrizable and noncompact. A complete compatible metric $w$ is not totally bounded, so there are $\epsilon>0$ and a countably infinite $\epsilon$-separated set $A\subseteq X$. This set is closed and discrete. Choose a bijection $\iota\colon \Nseq\to A$, and set, for $T\in\Tr$ and $x,y\in A$,
\begin{equation}\label{eq:complete-metric-transport}
 q_T(x,y)=\yocmpd{T}(\iota^{-1}(x),\iota^{-1}(y)).
\end{equation}
Then $\iota$ is an isometry from $(\Nseq,\yocmpd{T})$ onto $(A,q_T)$, {and $q_T\in\Met(A)$. For $T,U\in\Tr$,}
\begin{equation}\label{eq:complete-metric-transport-isometry}
 \yometdis{A}(q_T,q_U)=\yometdis{\Nseq}(\yocmpd{T},\yocmpd{U}),
\end{equation}
{so the parameter map} is continuous by \Cref{thm:complete-tree-reduction}. Choose the extensor
\[
 E\colon\CPM(A)\to\CPM(X)
\]
in the completely metrizable case of \Cref{prop:extensor}. For every $T\in\Tr$,
\begin{equation}\label{eq:complete-metric-extension}
 E(q_T)\in\Comp(X)
 \quad\Longleftrightarrow\quad
 q_T\in\Comp(A)
\end{equation}
by \eqref{eq:extensor-completeness}. Together with \eqref{eq:complete-tree-equivalence}, this shows that the continuous map
\[
 F\colon\Tr\to\Met(X),\qquad F(T)=E(q_T)
\]
satisfies
\begin{equation}\label{eq:complete-metric-reduction}
 F^{-1}(\Comp(X))=\WF.
\end{equation}
Hence $\Comp(X)$ is non-Borel relative to $\Met(X)$.

We have \ref{item:intro-complete-metric-ambient} $\Rightarrow$ \ref{item:intro-complete-metric-relative} by restriction of the ambient space. Also, \ref{item:intro-complete-metric-complete} implies \ref{item:intro-complete-metric-relative}, since a completely metrizable subspace of $\Met(X)$ is $G_\delta$ by \Cref{thm:complete-subspaces-gdelta}. If \ref{item:intro-complete-metric-relative} holds and $\Comp(X)\ne\emptyset$, then $X$ is completely metrizable, and the preceding construction excludes noncompactness. Thus \ref{item:intro-complete-metric-relative} implies \ref{item:intro-complete-metric-compact}.

If $X$ is compact, every compatible metric is complete. Hence $\Comp(X)=\Met(X)$, which is Borel in $\CPM(X)$ and completely metrizable by \Cref{thm:introduction}. If $\Comp(X)$ is empty, both properties hold as well. Therefore \ref{item:intro-complete-metric-compact} implies the other three conditions. Since emptiness of $\Comp(X)$ means that $X$ is not completely metrizable, the final condition is equivalent to $X$ being either not completely metrizable or compact.
\end{proof}

\begin{corollary}\label{rmk:bounded-comp}
Let
$X$
be a metrizable space.
Put
\[
 \Compbd{X}=\Comp(X)\cap\yometbd{X}.
\]
Then the following conditions are equivalent.
\begin{enumerate}[label=\textup{(B\arabic*)},ref=\textup{(B\arabic*)}]
\item\label{item:bounded-complete-metric-ambient}
$\Compbd{X}$
is Borel in
$\yocpmbd{X}$.
\item\label{item:bounded-complete-metric-relative}
$\Compbd{X}$
is Borel relative to
$\yometbd{X}$.
\item\label{item:bounded-complete-metric-complete}
$\Compbd{X}$
is completely metrizable.
\item\label{item:bounded-complete-metric-compact}
The space
$\Compbd{X}$
is empty or
$X$
is compact.
\item\label{item:bounded-complete-metric-existence}
The space
$X$
is not completely metrizable or is compact.
\end{enumerate}
\end{corollary}

\begin{proof}
The argument of \Cref{thm:introduction-comp} {also applies to the subspaces consisting of bounded metrics and pseudometrics.} The required tree family in \Cref{thm:complete-tree-reduction} is bounded, and the extensor in the completely metrizable case may be chosen to preserve boundedness as well as completeness.
\end{proof}

\section{Spaces of prescribed range ultrametrics}
\label{sec:ultrametric-components}
We translate the preceding Archimedean arguments into non-Archimedean ones by replacing sums with maxima and choosing the positive values in the metric constructions from $R$. In this way, we prove the non-Archimedean analogues of \Cref{thm:introduction,thm:introduction-components}, namely \Cref{thm:introduction-ultrametric,thm:introduction-ultrametric-components}, for spaces of prescribed range ultrametrics and their $R$-components. {The $R$-components are defined by the relation $\umetdis_X^R(d,e)<\infty$.}

\subsection{{The ultrametric $\umetdis_X^R$ and compatibility}}

For a characteristic range set $R$, we recall existence results for compatible $R$-valued ultrametrics. We then study {the extended ultrametric $\umetdis_X^R$ on $\UCPM{X}{R}$} and recall the $G_\delta$ property of $\UMet{X}{R}$ when $X$ is $\sigma$-compact.


\begin{proposition}\label{prop:ultrametric-existence}
Let
$R$
be characteristic and let
$X$
be ultrametrizable.
Then the following statements hold.
\begin{enumerate}[label=\textup{(E\arabic*)},ref=\textup{(E\arabic*)}]
\item\label{item:ultrametric-existence-bounded}
The space
$X$
admits a bounded compatible
$R$-valued ultrametric \cite[Proposition~2.14 and its proof]{ishiki-embedding}.
\item\label{item:ultrametric-existence-total}
If
$X$
is separable, then it admits a totally bounded compatible
$R$-valued ultrametric.
This is a consequence of \cite[Proposition~2.5]{ishiki-factorization} and \cite[Proposition~2.14]{ishiki-embedding}.
\item\label{item:ultrametric-existence-complete}
If
$X$
is completely metrizable, then it admits a complete compatible
$R$-valued ultrametric \cite[Proposition~2.17]{ishiki-embedding}.
\end{enumerate}
\end{proposition}

\begin{lemma}[{\cite[Section~1.1, item~7, p.~4]{perez-garcia-schikhof}}]\label{lem:ultrametric-completion-range}
Let
$R$
be a range set and let
$(X,d)$
be an
$R$-valued
ultrametric space,
and let
$(\widehat X,\widehat d)$
be its completion.
Then
\begin{equation}\label{eq:completion-distance-range}
 d(X^2)=\widehat d(\widehat X^2).
\end{equation}
In particular,
$\widehat d$
is
$R$-valued.
\end{lemma}


For characteristic $R$, ultrametrizable $X$, and $d,e\in\UCPM{X}{R}$, write
\begin{equation}\label{eq:ud-definition}
 \umetdis_X^R(d,e)=\inf\Bigl\{\epsilon\in R\setminus\set{0}\,\Bigm|
 \begin{array}{l}
 d(x,y)\leq e(x,y)\lor\epsilon,\\
 e(x,y)\leq d(x,y)\lor\epsilon
 \end{array}
 \text{ for all }x,y\in X\Bigr\},
\end{equation}
{where $\inf\emptyset=\infty$. The two inequalities require $d(x,y)=e(x,y)$ whenever either value exceeds $\epsilon$; equivalently, both values must be at most $\epsilon$ at every pair where they differ.} The topology induced by this extended ultrametric is metrizable. In any subspace of $\UCPM{X}{R}$, the equivalence classes under $\umetdis_X^R(p,q)<\infty$ are its \emph{$R$-components}.

\begin{lemma}\label{lem:threshold-agreement}
Let
$R$
be characteristic,
let
$X$
be ultrametrizable,
and let
$d$
and
$e$
be continuous
$R$-valued
pseudo-ultrametrics on
$X$.
If
$\epsilon>0$
and
$\umetdis_X^R(d,e)<\epsilon$,
then
for every
$x,y\in X$,
\[
 d(x,y)=e(x,y)
 \quad\text{whenever}\quad
 d(x,y)\geq\epsilon\ \text{or}\ e(x,y)\geq\epsilon.
\]
\end{lemma}

\begin{proof}
By definition of $\umetdis_X^R$, choose $\eta\in R\setminus\{0\}$ with $\eta<\epsilon$ such that, for all $x,y\in X$,
\begin{equation}\label{eq:threshold-comparison}
 \begin{aligned}
 d(x,y)&\leq e(x,y)\lor\eta,\\
 e(x,y)&\leq d(x,y)\lor\eta.
 \end{aligned}
\end{equation}
If $d(x,y)\geq\epsilon$, the first inequality forces $d(x,y)\leq e(x,y)$. The second then forces the reverse inequality. Hence $d(x,y)=e(x,y)$. Interchanging $d$ and $e$ treats the other case.
\end{proof}

For $u\in\UMet{X}{R}$, set
\[
 \begin{aligned}
 \youpclq{u}&=
 \set*{p\in\UCPM{X}{R}\middle|\umetdis_X^R(p,u)<\infty},\\
 \youcclq{u}&=\youpclq{u}\cap\UMet{X}{R}.
 \end{aligned}
\]
Thus $\youcclq{u}$ is the $R$-component of $u$, in the terminology of \cite[Definition~2.12]{ishiki-uda-borel}.

The space $\UCPM{X}{R}$ is complete for $\min\{1,\umetdis_X^R\}$, and each ambient $R$-component $\youpclq{u}$ is complete for $\umetdis_X^R$, as follows.

\begin{proposition}[{Consequence of \cite[Lemma~2.4]{ishiki-baire}}]\label{lem:ultrametric-ambient-component}
Let
$R$
be characteristic and let
$X$
be ultrametrizable.
Then
$\UCPM{X}{R}$
is complete for
$\min\{1,\umetdis_X^R\}$.
Moreover, for every
$u\in\UMet{X}{R}$,
the space
$\youpclq{u}$
is clopen in
$\UCPM{X}{R}$
and is complete for
$\umetdis_X^R$.
\end{proposition}


For $\sigma$-compact $X$, we use the following $G_\delta$ property. It also implies that $\youcclq{u}$ is $G_\delta$ in $\youpclq{u}$.

\begin{proposition}[{Consequence of \cite[Lemma~7.4]{ishiki-embedding} and \cite[Section~5.2]{ishiki-baire}}]\label{lem:ultrametric-gdelta}
Let
$R$
be characteristic and let
$X$
be a
$\sigma$-compact
ultrametrizable space.
Then
$\UMet{X}{R}$
is
$G_\delta$
in
$\UCPM{X}{R}$.
\end{proposition}


\subsection{Extension and bounded continuous families}

For a nonempty closed subspace $A$ of an ultrametrizable space $X$ and characteristic $R$, the following extensor sends continuous $R$-valued ultrametrics on $A$ to those on $X$. {It is isometric for $\umetdis_A^R$ and $\umetdis_X^R$, and preserves compatibility in both directions.} We use it to extend the bounded continuous families constructed below.

\begin{theorem}\label{lem:ultrametric-max-transfer}
Let
$R$
be characteristic and let
$A$
be a nonempty closed subspace of an ultrametrizable space
$X$.
Then there exist
$h_1\in R\setminus\set{0}$
and a map
$\Sigma$
from the continuous
$R$-valued
ultrametrics on
$A$
to the continuous
$R$-valued
ultrametrics on
$X$
with the following properties.
For every continuous
$R$-valued
ultrametric
$q$
on
$A$,
\begin{equation}\label{eq:ultrametric-extension-restriction}
 \Sigma(q)\restr A^2=q,
\end{equation}
and
\begin{equation}\label{eq:ultrametric-extension-compatibility}
 \Sigma(q)\in\UMet{X}{R}
 \quad\Longleftrightarrow\quad q\in\UMet{A}{R}.
\end{equation}
For any two such ultrametrics
$q_0$
and
$q_1$,
\begin{equation}\label{eq:ultrametric-extension-isometry}
 \umetdis_X^R(\Sigma(q_0),\Sigma(q_1))
 =\umetdis_A^R(q_0,q_1).
\end{equation}
For every such ultrametric
$q$
and every
$h_0\in R$,
if
$q\leq h_0$
on
$A^2$,
then
$\Sigma(q)\leq h_0\lor h_1$
on
$X^2$.
If
$X$
is completely metrizable,
then
$\Sigma$
and
$h_1$
may be chosen so that,
for every
$q\in\UMet{A}{R}$,
\begin{equation}\label{eq:ultrametric-extension-completeness}
 \Sigma(q)\in\UComp{X}{R}
 \quad\Longleftrightarrow\quad q\in\UComp{A}{R}.
\end{equation}
\end{theorem}

\begin{proof}
We employ the construction in \cite[Theorem~3.1 and Remark~3.1]{ishiki-factorization}. It gives an ultrametrizable space $\Lambda$, a point $\omega\in\Lambda$, a retraction $\alpha\colon X\to A$, and a continuous $\pi\colon X\to\Lambda$ such that
\begin{equation}\label{eq:ultrametric-factorization}
 (\alpha,\pi)\colon X\to A\times\Lambda
\end{equation}
is a topological embedding with closed image and $(\alpha(a),\pi(a))=(a,\omega)$ for $a\in A$. If $X$ is completely metrizable, $\Lambda$ can also be chosen completely metrizable \cite[Proposition~2.9 and the proof of Theorem~4.7(N1)]{ishiki-factorization}. By \Cref{prop:ultrametric-existence}, choose a compatible $R$-valued ultrametric $\theta_0$ on $\Lambda$, complete in that case. Fix $h_1\in R\setminus\set{0}$ and, for $\lambda,\kappa\in\Lambda$, put
\begin{equation}\label{eq:ultrametric-truncation}
 \theta(\lambda,\kappa)=\min\set{\theta_0(\lambda,\kappa),h_1}.
\end{equation}
For $0<\epsilon<h_1$,
\[
 \theta(\lambda,\kappa)<\epsilon
 \quad\Longleftrightarrow\quad
 \theta_0(\lambda,\kappa)<\epsilon,
\]
so $\theta$ and $\theta_0$ induce the same topology and have the same Cauchy sequences. In particular, if $\theta_0$ is complete, then $\theta$ is complete. For a continuous $R$-valued ultrametric $q$ on $A$ and $x,y\in X$, define
\begin{equation}\label{eq:ultrametric-extension-formula}
 \Sigma(q)(x,y)
 =q(\alpha(x),\alpha(y))\lor \theta(\pi(x),\pi(y)).
\end{equation}
This function is continuous and satisfies the strong triangle inequality. If its value at $(x,y)$ is zero, both coordinates of $(\alpha,\pi)$ agree at $x$ and $y$, so $x=y$. Thus $\Sigma(q)$ is a continuous $R$-valued ultrametric, whether or not $q$ is compatible.

The identities $\alpha(a)=a$ and $\pi(a)=\omega$ prove \eqref{eq:ultrametric-extension-restriction}. If $q$ is compatible, the maximum product ultrametric on $A\times\Lambda$ induces the product topology, and its pullback induces the topology of $X$. Conversely, compatibility of $\Sigma(q)$ implies that of its restriction $q$. This proves \eqref{eq:ultrametric-extension-compatibility}. For isometry, let $q_0,q_1$ be continuous $R$-valued ultrametrics on $A$. If a positive $\eta\in R$ satisfies $q_0\leq q_1\lor\eta$ and $q_1\leq q_0\lor\eta$, then taking maxima in \eqref{eq:ultrametric-extension-formula} yields
\begin{equation}\label{eq:ultrametric-extension-threshold}
 \begin{aligned}
 \Sigma(q_0)&\leq\Sigma(q_1)\lor\eta,\\
 \Sigma(q_1)&\leq\Sigma(q_0)\lor\eta.
 \end{aligned}
\end{equation}
Taking the infimum over such $\eta$ proves one inequality in \eqref{eq:ultrametric-extension-isometry}. Restriction to $A^2$ proves the other. Finally, $q\leq h_0$ implies $\Sigma(q)\leq h_0\lor h_1$ directly from the defining formula.

If $X$ is completely metrizable, choose $\Lambda$ and the complete $\theta$ as above. For $q\in\UComp{A}{R}$, the maximum product metric on $A\times\Lambda$ is complete. Its restriction to the closed image $(\alpha,\pi)(X)$ is complete, so
\[
 q\in\UComp{A}{R}\quad\Longrightarrow\quad
 \Sigma(q)\in\UComp{X}{R}.
\]
Conversely, if $\Sigma(q)$ is complete and compatible, its restriction to the topologically closed subspace $A$ is complete. This proves \eqref{eq:ultrametric-extension-completeness}.
\end{proof}

We now construct the bounded tree family on the same topological sum as in \Cref{thm:tree-reduction}.

\begin{lemma}\label{lem:ultrametric-tree-reduction}
Let
$R$
be characteristic.
Let $\TreeBoundarySpace=P\sqcup V$, with $P=\Baire$ and $V=\set{\yov{s}\mid s\in\Nseq}$, be the topological sum defined in \Cref{sec:trees}.
Then there exist a positive
$h_0\in R$
and a map
$T\mapsto\youu{T}$
from
$\Tr$
to
$\UCPM{\TreeBoundarySpace}{R}$
that is continuous for
$\umetdis_{\TreeBoundarySpace}^R$
such that,
for every
$T\in\Tr$,
the function
$\youu{T}$
is an
$R$-valued ultrametric bounded by
$h_0$
and
\[
 \youu{T}\in\UMet{\TreeBoundarySpace}{R}\quad\Longleftrightarrow\quad T\in\WF.
\]
\end{lemma}

\begin{proof}
Choose a strictly decreasing null sequence $\{r(n)\}_{n\in\yoomega}$ in $R\setminus\set{0}$. For $x,y\in\Baire$, set
\[
 \yourho(x,y)=
 \begin{cases}
  0,&x=y,\\
  r(\min\{n\in\yoomega\mid x(n)\ne y(n)\}),&x\ne y.
 \end{cases}
\]
Use the bijection $\yotreeNu$ fixed in \Cref{def:tree-space}. For $s\in\Nseq$, use its zero extension $\youext{s}$ from \eqref{eq:zero-extension} and the weight
\[
 \youc{s}=r(\yotreeNu(s)).
\]
Fix $T\in\Tr$. For each $s\in\Nseq$, define
\begin{equation}\label{eq:ultrametric-tree-penalties}
 \youa{T}{s}=\max\left(\{\youc{s}\}\cup
  \set*{\youc{u}\middle|u\preccurlyeq s,\ u\notin T}\right).
\end{equation}
For $x,y\in\Baire$ and $s,t\in\Nseq$, define $\youu{T}$ by
\[
 \youu{T}(x,y)=\yourho(x,y),
\]
\begin{equation}\label{eq:ultrametric-tree-mixed-distance}
 \youu{T}(\yov{s},x)=\youu{T}(x,\yov{s})
 =\youa{T}{s}\lor\yourho(\youext{s},x),
\end{equation}
and
\begin{equation}\label{eq:ultrametric-tree-distance}
 \youu{T}(\yov{s},\yov{t})=
 \begin{cases}
  0,&s=t,\\
  \youa{T}{s}\lor\yourho(\youext{s},\youext{t})\lor\youa{T}{t},&s\ne t
 \end{cases}.
\end{equation}
The maximum formulas define a continuous $R$-valued ultrametric bounded by $r(0)$. Fix $x\in\Baire$. Its distance from the discrete summand is
\begin{equation}\label{eq:ultrametric-tree-cross-distance}
 \inf_{s\in\Nseq}
 \bigl(\youa{T}{s}\lor\yourho(\youext{s},x)\bigr).
\end{equation}
For nonnegative $a,b$,
\begin{equation}\label{eq:max-sum-comparison}
 a\lor b\leq a+b\leq2(a\lor b).
\end{equation}
Applying this inequality with $a=\youa{T}{s}$ and $b=\yourho(\youext{s},x)$ shows that \eqref{eq:ultrametric-tree-cross-distance} vanishes exactly when
\[
 \inf_{s\in\Nseq}\bigl(\youa{T}{s}+\yourho(\youext{s},x)\bigr)=0.
\]
The initial-segment argument in \Cref{thm:tree-reduction} applies to $\{r(n)\}_{n\in\yoomega}$ and the weights, since $r(n)$ decreases to zero and only finitely many weights exceed any fixed positive number. It shows that all the distances \eqref{eq:ultrametric-tree-cross-distance} are positive exactly when $T$ is well-founded. Therefore
\begin{equation}\label{eq:ultrametric-tree-compatibility}
 \youu{T}\in\UMet{\TreeBoundarySpace}{R}\quad\Longleftrightarrow\quad T\in\WF.
\end{equation}

For continuity, given $\epsilon>0$, choose $\eta\in R\setminus\set{0}$ with $\eta<\epsilon$ and define the finite set
\[
 H_\eta=\set*{s\in\Nseq\middle|\youc{s}>\eta}.
\]
Choose $N\in\yoomega$ such that $\yotreeNu(s)\leq N$ for every $s\in H_\eta$. If $T,U\in\Tr$ satisfy $\Delta(T,U)<2^{-N}$, then \eqref{eq:tree-parameter-metric} ensures that they agree on $H_\eta$. Thus \eqref{eq:ultrametric-tree-penalties} implies, for every $s\in\Nseq$,
\begin{equation}\label{eq:ultrametric-tree-penalty-continuity}
 \begin{aligned}
 \youa{T}{s}&\leq\youa{U}{s}\lor\eta,\\
 \youa{U}{s}&\leq\youa{T}{s}\lor\eta.
 \end{aligned}
\end{equation}
{For all $a,b\in\TreeBoundarySpace$, taking maxima in \eqref{eq:ultrametric-tree-mixed-distance} and \eqref{eq:ultrametric-tree-distance} yields}
\begin{equation}\label{eq:ultrametric-tree-distance-continuity}
 \begin{aligned}
 \youu{T}(a,b)&\leq\youu{U}(a,b)\lor\eta,\\
 \youu{U}(a,b)&\leq\youu{T}(a,b)\lor\eta
 \end{aligned}.
\end{equation}
{} Thus $\umetdis_{\TreeBoundarySpace}^R(\youu{T},\youu{U})\leq\eta<\epsilon$, which proves continuity of the parameter map on $(\Tr,\Delta)$. The bound $h_0=r(0)$ completes the construction.
\end{proof}

The next family uses a compact subset of a completion as its parameter space.

\begin{lemma}\label{lem:ultrametric-boundary-reduction}
Let
$R$
be characteristic,
let
$(L,\rho)$
be a compact
$R$-valued ultrametric space,
and let
$Y$
be a dense non-Borel subspace of
$L$.
Put
\[
 h_0=\max_{x,y\in L}\rho(x,y).
\]
Then
$h_0\in R\setminus\set{0}$,
and there exist a nonempty closed subspace
$A\subseteq Y$,
a compact set
$K\subseteq L$
for which
$K\setminus Y$
is non-Borel in
$K$,
and a map
$z\mapsto e_z$
from
$K$
to
$\UCPM{A}{R}$
with the following properties.
For every
$z\in K$,
the function
$e_z$
is an
$R$-valued ultrametric bounded by
$h_0$,
and
\[
 e_z\in\UMet{A}{R}\quad\Longleftrightarrow\quad z\in K\setminus Y.
\]
For every
$z,w\in K$,
\[
 \umetdis_A^R(e_z,e_w)\leq\rho(z,w).
\]
\end{lemma}

\begin{proof}
Set
\[
 \begin{aligned}
 I&=\Int_L(Y), & B&=Y\setminus I,\\
 C&=\cl_L(B), & Z&=C\setminus Y.
 \end{aligned}
\]
As in \Cref{prop:nonborel-completion}, $Z$ is non-Borel in $C$. Fix $a\in Z$. By using a decreasing null sequence in $R$, choose a compact $K\subseteq C$ at positive distance from $a$ such that $Z_K=Z\cap K$ is non-Borel in $K$. {Choose a sequence $\{a_m\}_{m\in\yoomega}$ of distinct points of $Y$ tending to $a$ such that, for every $m\in\yoomega$,}
\begin{equation}\label{eq:ultrametric-boundary-approach}
 \rho(a_m,a)<\inf_{z\in K}\rho(a,z),
\end{equation}
{and set}
\[
 A_0=\set{a_m\mid m\in\yoomega},\qquad B_0=B\cap K.
\]
We have $K\cap Y=B_0$ and $K=B_0\cup Z_K$, since $C\cap Y=B$. Also $B_0\ne\emptyset$, since $Z_K$ is non-Borel in $K$. The sets $A_0$ and $B_0$ are closed in $Y$. For $A_0$, this follows because its only accumulation point in $L$ is $a\notin Y$. The strong triangle inequality and \eqref{eq:ultrametric-boundary-approach} imply, for $m\in\yoomega$ and $y\in B_0$,
\begin{equation}\label{eq:ultrametric-boundary-separation}
 \rho(a_m,y)=\rho(a,y)\geq\inf_{z\in K}\rho(a,z)>0.
\end{equation}
Thus $A=A_0\cup B_0$ is a nonempty closed subspace of $Y$ with the topology of $A_0\sqcup B_0$. For $z\in K$, let $e_z$ agree with $\rho$ on $A_0^2$ and $B_0^2$, and define, for $x\in A_0$ and $y\in B_0$,
\begin{equation}\label{eq:ultrametric-boundary-cross-distance}
 e_z(x,y)=e_z(y,x)=\rho(x,a)\lor\rho(y,z).
\end{equation}
The strong triangle inequality for $\rho$ shows that $e_z$ is an $R$-valued ultrametric. For the fixed topology of $A_0\mathbin{\sqcup}B_0$, the defining formulas are continuous on each product of summands. Thus $e_z\in\UCPM{A_0\mathbin{\sqcup}B_0}{R}$ for every $z\in K$. For distinct $z,w\in K$, put $\delta=\rho(z,w)\in R\setminus\set{0}$. The functions agree on $A_0^2\cup B_0^2$, and the strong triangle inequality implies $e_z(x,y)\leq e_w(x,y)\lor\delta$ and $e_w(x,y)\leq e_z(x,y)\lor\delta$ on $A_0\times B_0$. Symmetry and \eqref{eq:ud-definition} yield
\begin{equation}\label{eq:ultrametric-boundary-continuity}
 \umetdis_{A_0\mathbin{\sqcup}B_0}^R(e_z,e_w)\leq\rho(z,w),
\end{equation}
which also holds for $z=w$. Hence the parameter map is continuous. Moreover,
\begin{equation}\label{eq:ultrametric-boundary-compatibility}
 e_z\in\UMet{A_0\mathbin{\sqcup}B_0}{R}
 \quad\Longleftrightarrow\quad z\in Z_K.
\end{equation}
If $z\in B_0$, the sequence $a_m$ converges to $z$ for $e_z$, contrary to the fixed topology. If $z\in Z_K$, each point has positive distance from the other summand, so $e_z$ generates the fixed topology.

Since $Z_K=K\setminus Y$, \eqref{eq:ultrametric-boundary-compatibility} is the required equivalence. The metric $\rho$ attains its maximum $h_0\in R$ on compact $L^2$. The space $L$ has at least two points, since it contains the non-Borel subspace $Y$, so $h_0>0$. Every value defining $e_z$ is at most $h_0$. This proves the common bound and completes the lemma.
\end{proof}

We combine the two families with the nonseparable construction of Ishiki and Uda. A common upper bound in $R$ will be needed when placing them in a prescribed $R$-component.

\begin{proposition}\label{lem:bounded-ultrametric-obstruction}
Let
$R$
be characteristic.
If an ultrametrizable space
$Y$
is not
$\sigma$-compact,
then there exist a Polish space
$K$,
a non-Borel set
$S\subseteq K$,
a positive
$h\in R$,
and a map
\[
 q\colon K\to\UCPM{Y}{R}
\]
that is continuous for
$\umetdis_Y^R$
such that,
for every
$t\in K$,
$q(t)$
is a continuous
$R$-valued
ultrametric bounded by
$h$
and
\[
 q(t)\in\UMet{Y}{R}\quad\Longleftrightarrow\quad t\in S.
\]
\end{proposition}

\begin{proof}
It suffices to construct the family on a nonempty closed $A\subseteq Y$. Indeed, let $p\colon K\to\UCPM{A}{R}$ have the required properties with bound $h_0$. {Choose $\Sigma$ and $h_1$ from \Cref{lem:ultrametric-max-transfer}. For $t\in K$, set}
\[
 q(t)=\Sigma(p(t)).
\]
{Equations} \eqref{eq:ultrametric-extension-compatibility} and \eqref{eq:ultrametric-extension-isometry} preserve the required compatibility condition and continuity. {For all $t\in K$ and $x,y\in Y$, the extension bound gives}
\[
 q(t)(x,y)\leq h_0\lor h_1.
\]
{Thus} $h=h_0\lor h_1\in R\setminus\set{0}$ is a common bound.

If $Y$ is nonseparable, take $A$ to be a closed discrete subspace of $Y$ of cardinality $\aleph_1$. By \cite[Lemmas~4.1--4.2]{ishiki-uda-borel}, there is a continuous map
\[
 q_A\colon 2^{\yoomega}\to\UCPM{A}{R}
\]
{for $\umetdis_A^R$.} Its values are continuous ultrametrics bounded by the first term of a fixed null sequence in $R$. The set of parameters for which they are compatible is non-Borel in $2^{\yoomega}$. This is the required family on $A$.

Assume now that $Y$ is separable. Choose a totally bounded compatible $R$-valued ultrametric $\rho$ by \Cref{prop:ultrametric-existence}\ref{item:ultrametric-existence-total}, and let $L$ be its compact completion. The extended metric remains $R$-valued by \Cref{lem:ultrametric-completion-range}. If $Y$ is Borel in $L$, its failure to be $\sigma$-compact and \Cref{thm:hurewicz} imply that it contains a closed copy of $\Baire$. As in \Cref{cor:closed-baire}, choose a closed $A_1\subseteq Y$ and a homeomorphism $\iota\colon \TreeBoundarySpace\to A_1$, with $\TreeBoundarySpace$ as in \Cref{lem:ultrametric-tree-reduction}. Transport the family of that lemma by setting, for $T\in\Tr$ and $x,y\in A_1$,
\begin{equation}\label{eq:ultrametric-tree-transport}
 q_T(x,y)=\youu{T}(\iota^{-1}(x),\iota^{-1}(y)).
\end{equation}
For $T,U\in\Tr$,
\begin{equation}\label{eq:ultrametric-tree-transport-isometry}
 \umetdis_{A_1}^R(q_T,q_U)=\umetdis_{\TreeBoundarySpace}^R(\youu{T},\youu{U}),
\end{equation}
so the map is continuous and retains the bound $h_0$. Since $\iota$ is a homeomorphism, $q_T$ is compatible precisely when $T\in\WF$. Thus $K=\Tr$ and $S=\WF$ have the required properties.

If $Y$ is non-Borel in $L$, use \Cref{lem:ultrametric-boundary-reduction}. Its parameter space $K$ is compact metrizable, hence Polish, and $S=K\setminus Y$ is non-Borel. The lemma supplies the family on a nonempty closed subspace of $Y$.
\end{proof}

\subsection{The component and global characterizations}

We prove the $R$-component characterization \Cref{thm:introduction-ultrametric-components}, followed by the global characterization \Cref{thm:introduction-ultrametric}. For fixed $u\in\UMet{X}{R}$ and $\sigma$-compact $X$, \Cref{lem:ultrametric-gdelta} makes $\youcclq{u}$ completely metrizable. When $X$ is not $\sigma$-compact, we use the bounded families of \Cref{lem:bounded-ultrametric-obstruction} to prove non-Borelness of $\youcclq{u}$.

\begin{proof}[Proof of \Cref{thm:introduction-ultrametric-components}]
By \Cref{lem:ultrametric-ambient-component}, $\youpclq{u}$ is clopen in $\UCPM{X}{R}$. Thus $\youcclq{u}$ is Borel in one of these spaces if and only if it is Borel in the other. If $\youcclq{u}$ is completely metrizable, it is $G_\delta$ in $\youpclq{u}$ by \Cref{thm:complete-subspaces-gdelta}. This proves \ref{item:intro-ultrametric-component-complete} $\Rightarrow$ \ref{item:intro-ultrametric-component-borel}.

If $X$ is $\sigma$-compact, \Cref{lem:ultrametric-gdelta} makes $\youcclq{u}=\youpclq{u}\cap\UMet{X}{R}$ a $G_\delta$ subset of the complete space $\youpclq{u}$. It is therefore completely metrizable by \Cref{thm:complete-subspaces-gdelta}, proving \ref{item:intro-ultrametric-component-sigma} $\Rightarrow$ \ref{item:intro-ultrametric-component-complete}.

Assume now that $X$ is not $\sigma$-compact. We show that $\youcclq{u}$ is non-Borel. If $X$ is nonseparable, this is \cite[Theorem~4.6(U1)]{ishiki-uda-borel}. Otherwise fix $x_0\in X$. The set
\[
 \set*{u(x_0,x)\middle|x\in X}\subseteq R
\]
has a countable cofinal subset. Hence $X$ is a countable union of closed $u$-balls about $x_0$ with radii in $R$. At least one is not $\sigma$-compact. Choose its radius $s\in R\setminus\set{0}$, and write
\[
 B=\set*{x\in X\middle|u(x_0,x)\leq s}.
\]
This ball is clopen.

For $x,y\in X$, define
\begin{equation}\label{eq:ultrametric-collapsed-pseudometric}
 r_B(x,y)=
 \begin{cases}
  0,&x,y\in B,\\
  u(x,y),&\text{otherwise}.
 \end{cases}
\end{equation}
If $x\notin B$, the value $u(x,b)$ is independent of $b\in B$ and exceeds $s$. Thus $r_B$ is the quotient pseudo-ultrametric obtained by identifying all points of $B$. Since $B$ is clopen, it is continuous. {Moreover,}
\begin{equation}\label{eq:ultrametric-collapsed-bound}
 \umetdis_X^R(r_B,u)\leq s.
\end{equation}

Apply \Cref{lem:bounded-ultrametric-obstruction} to $B$. Let $K$ be the resulting Polish parameter space, $S\subseteq K$ the non-Borel set, and $q_t=q(t)$ the continuous family of ultrametrics bounded by $h_0\in R$. Thus $q_t$ is compatible exactly when $t\in S$. Choose $\Sigma$ and $h_1$ from \Cref{lem:ultrametric-max-transfer} {for $B\subseteq X$. For every $t\in K$,}
\begin{equation}\label{eq:ultrametric-family-extension-bound}
 \Sigma(q_t)\leq h_0\lor h_1.
\end{equation}
{Put $h=h_0\lor h_1$. For $t\in K$, define}
\begin{equation}\label{eq:ultrametric-component-transfer}
 F(t)=r_B\lor\Sigma(q_t).
\end{equation}
For $t,v\in K$, taking a maximum with the fixed $r_B$ and using \eqref{eq:ultrametric-extension-isometry} yields
\begin{equation}\label{eq:ultrametric-component-continuity}
 \umetdis_X^R(F(t),F(v))\leq\umetdis_B^R(q_t,q_v),
\end{equation}
so $F$ is continuous. {For every $t\in K$, the bound on $\Sigma(q_t)$ also implies}
\begin{equation}\label{eq:ultrametric-component-distance}
 \umetdis_X^R(F(t),r_B)\leq h.
\end{equation}
Combining this with \eqref{eq:ultrametric-collapsed-bound}, we obtain $\umetdis_X^R(F(t),u)\leq h\lor s<\infty$. Thus $F$ takes values in $\youpclq{u}$.

For each $t\in K$, the vanishing of $r_B$ on $B^2$ and the extension property imply
\begin{equation}\label{eq:ultrametric-component-restriction}
 F(t)\restr B^2=q_t.
\end{equation}
If $q_t$ is compatible, so is $\Sigma(q_t)$. The continuous ultrametric $F(t)$ dominates it and is therefore compatible. Conversely, compatibility of $F(t)$ implies that of its restriction $q_t$. Hence
\begin{equation}\label{eq:ultrametric-component-reduction}
 F^{-1}(\youcclq{u})=S.
\end{equation}
Since $S$ is non-Borel and $F$ is continuous, $\youcclq{u}$ is non-Borel in $\youpclq{u}$, as required.
\end{proof}

We now prove \Cref{thm:introduction-ultrametric}. Its non-Borelness assertion for $\UMet{X}{R}$ in $\UCPM{X}{R}$ follows from the $R$-component characterization \Cref{thm:introduction-ultrametric-components}.

\begin{proof}[Proof of \Cref{thm:introduction-ultrametric}]
If $X$ is $\sigma$-compact, \Cref{lem:ultrametric-gdelta} makes $\UMet{X}{R}$ a $G_\delta$ subset of $\UCPM{X}{R}$. The latter is complete for $\min\{1,\umetdis_X^R\}$ by \Cref{lem:ultrametric-ambient-component}. Thus \Cref{thm:complete-subspaces-gdelta} makes $\UMet{X}{R}$ completely metrizable for the specified topology.

If $X$ is not $\sigma$-compact, choose $u\in\UMet{X}{R}$ by \Cref{prop:ultrametric-existence}\ref{item:ultrametric-existence-bounded}. The ambient $R$-component $\youpclq{u}$ is clopen, and
\[
 \youcclq{u}=\UMet{X}{R}\cap\youpclq{u}
\]
is non-Borel there by \Cref{thm:introduction-ultrametric-components}. Therefore $\UMet{X}{R}$ is non-Borel in $\UCPM{X}{R}$. This proves that Borelness forces $\sigma$-compactness.

Finally, by \Cref{thm:complete-subspaces-gdelta}, complete metrizability of $\UMet{X}{R}$ implies that it is $G_\delta$ in the metrizable ambient space $\UCPM{X}{R}$ and hence Borel. All three conditions are therefore equivalent.
\end{proof}

\section{Complete compatible ultrametrics}\label{sec:complete-ultrametrics}
{We characterize Borelness and complete metrizability of the space $\UComp{X}{R}$ by compactness of $X$ when $\UComp{X}{R}$ is nonempty (\Cref{thm:introduction-ucomp}).} The criterion of \Cref{lem:discrete-tree-completeness} applies when the numbers $r(n)$ and the weights are all chosen in $R$.

\subsection{Ultrametrics on a countable discrete space}

Equip $\Nseq$ with the discrete topology. For characteristic $R$, choose the positive numbers $r(n)$ and the weights in \Cref{lem:discrete-tree-completeness} from $R$. We obtain a map $T\mapsto\youcompd{T}\in\UMet{\Nseq}{R}$ that is continuous for $\umetdis_{\Nseq}^R$ and whose complete members correspond exactly to well-founded trees.

\begin{theorem}\label{thm:ucomplete-tree-reduction}
Let
$R$
be characteristic and equip
$\Nseq$
with the discrete topology.
Then there is a map
\[
 \Psi\colon\Tr\to\UMet{\Nseq}{R},
 \qquad T\longmapsto\youcompd{T}
\]
that is continuous for
$\umetdis_{\Nseq}^R$.
For every
$T\in\Tr$,
$\youcompd{T}$
is bounded and satisfies
\begin{equation}\label{eq:complete-ultrametric-equivalence}
 \youcompd{T}\in\UComp{\Nseq}{R}\quad\Longleftrightarrow\quad T\in\WF.
\end{equation}
Consequently,
$\UComp{\Nseq}{R}$
is {non-Borel} even relative to
$\UMet{\Nseq}{R}$.
\end{theorem}

\begin{proof}
Choose a strictly decreasing null sequence $\{r(n)\}_{n\in\yoomega}$ in $R\setminus\set{0}$, and use the bijection $\yotreeNu$ fixed in \Cref{def:tree-space}. {For $s\in\Nseq$, define the weight}
\begin{equation}\label{eq:complete-ultrametric-weights}
 \youcompc{s}=r(\yotreeNu(s)).
\end{equation}
{For every} $\epsilon>0$, only finitely many weights are at least $\epsilon$. Apply \Cref{lem:discrete-tree-completeness} with this sequence $\{r(n)\}_{n\in\yoomega}$ and these weights. Denote the resulting ultrametric by $\youcompd{T}$ and the quantity in \eqref{eq:discrete-tree-penalties} by $\youcompa{T}{s}$. Every distance in \eqref{eq:discrete-tree-distance} is a finite maximum of elements of $R$. Thus $\youcompd{T}$ is $R$-valued, bounded by $r(0)$, and compatible with the discrete topology.

For every $\epsilon>0$, choose $\eta\in R\setminus\set{0}$ with $\eta<\epsilon$. Define the finite set
\[
 H_\eta=\set*{u\in \Nseq\middle|\youcompc{u}>\eta}.
\]
Choose $N\in\yoomega$ such that $\yotreeNu(u)\leq N$ for every $u\in H_\eta$. If $T,U\in\Tr$ satisfy $\Delta(T,U)<2^{-N}$, then \eqref{eq:tree-parameter-metric} ensures that they agree on $H_\eta$. Hence, for every $s\in \Nseq$,
\begin{equation}\label{eq:complete-ultrametric-penalty-continuity}
 \begin{aligned}
 \youcompa{T}{s}&\leq\youcompa{U}{s}\lor\eta,\\
 \youcompa{U}{s}&\leq\youcompa{T}{s}\lor\eta.
 \end{aligned}
\end{equation}
Substitution in \eqref{eq:discrete-tree-distance} yields, on $\Nseq^2$,
\begin{equation}\label{eq:complete-ultrametric-distance-continuity}
 \begin{aligned}
 \youcompd{T}&\leq\youcompd{U}\lor\eta,\\
 \youcompd{U}&\leq\youcompd{T}\lor\eta.
 \end{aligned}
\end{equation}
The definition of $\umetdis_{\Nseq}^R$ now implies $\umetdis_{\Nseq}^R(\youcompd{T},\youcompd{U})\leq\eta<\epsilon$. This proves continuity of $\Psi$ on $(\Tr,\Delta)$.

By \Cref{lem:discrete-tree-completeness}, equation \eqref{eq:complete-ultrametric-equivalence} holds. Thus the parameters for which the ultrametric is complete form exactly $\WF$. Since $\Psi$ is continuous and $\WF$ is non-Borel by \Cref{thm:wf-nonborel}, the asserted relative non-Borelness follows.
\end{proof}

\subsection{The characterization of complete ultrametrics}

To prove \Cref{thm:introduction-ucomp} for noncompact completely metrizable $X$, transfer the family of \Cref{thm:ucomplete-tree-reduction} to a countably infinite closed discrete subspace of $X$. The extensor of \Cref{lem:ultrametric-max-transfer} extends it to $X$ while preserving completeness, so the resulting continuous family proves non-Borelness of $\UComp{X}{R}$ relative to $\UMet{X}{R}$.

\begin{proof}[Proof of \Cref{thm:introduction-ucomp}]
Assume that $X$ is noncompact and admits a complete compatible $R$-valued ultrametric $w$. Since $(X,w)$ is complete but not totally bounded, it contains a countably infinite $\epsilon$-separated set $A$ for some $\epsilon>0$. The set $A$ is closed and discrete. Choose a bijection $\iota\colon \Nseq\to A$, {and, for $T\in\Tr$ and $x,y\in A$, transport the family of \Cref{thm:ucomplete-tree-reduction} by}
\begin{equation}\label{eq:complete-ultrametric-transport}
 q_T(x,y)=\youcompd{T}(\iota^{-1}(x),\iota^{-1}(y)).
\end{equation}
{Then} $q_T\in\UMet{A}{R}$ and $(A,q_T)$ is isometric to $(\Nseq,\youcompd{T})$. For $T,U\in\Tr$,
\begin{equation}\label{eq:complete-ultrametric-transport-isometry}
 \umetdis_A^R(q_T,q_U)=\umetdis_{\Nseq}^R(\youcompd{T},\youcompd{U}),
\end{equation}
so $T\mapsto q_T$ is continuous.

Choose the map that preserves completeness
\[
 \Sigma\colon\UMet{A}{R}\to\UMet{X}{R}
\]
from \Cref{lem:ultrametric-max-transfer}.
For each
$T\in\Tr$,
the following conditions are equivalent.
\begin{enumerate}[label=\textup{(T\arabic*)},ref=\textup{(T\arabic*)}]
\item\label{item:complete-ultrametric-transfer-image}
$\Sigma(q_T)\in\UComp{X}{R}$.
\item\label{item:complete-ultrametric-transfer-restriction}
$q_T\in\UComp{A}{R}$.
\item\label{item:complete-ultrametric-transfer-tree}
$T\in\WF$.
\end{enumerate}
Conditions \ref{item:complete-ultrametric-transfer-image} and \ref{item:complete-ultrametric-transfer-restriction} are equivalent by \eqref{eq:ultrametric-extension-completeness}. Restriction to the closed subspace $A$ proves the forward direction. The isometry $\iota\colon(\Nseq,\youcompd{T})\to(A,q_T)$ and \eqref{eq:complete-ultrametric-equivalence} identify \ref{item:complete-ultrametric-transfer-restriction} with \ref{item:complete-ultrametric-transfer-tree}. Hence the continuous map
\[
 F\colon\Tr\to\UMet{X}{R},\qquad F(T)=\Sigma(q_T)
\]
satisfies
\begin{equation}\label{eq:complete-ultrametric-reduction}
 F^{-1}(\UComp{X}{R})=\WF.
\end{equation}
The non-Borelness of $\WF$ implies that $\UComp{X}{R}$ is non-Borel relative to $\UMet{X}{R}$.

Clearly \ref{item:intro-complete-ultrametric-ambient} implies \ref{item:intro-complete-ultrametric-relative}. Also, \ref{item:intro-complete-ultrametric-complete} implies \ref{item:intro-complete-ultrametric-relative}, since \Cref{thm:complete-subspaces-gdelta} makes a completely metrizable subspace $G_\delta$ in $\UMet{X}{R}$. {Assume \ref{item:intro-complete-ultrametric-relative} and $\UComp{X}{R}\ne\emptyset$. Then $X$ is completely metrizable, and \eqref{eq:complete-ultrametric-reduction} excludes noncompact $X$.} Thus \ref{item:intro-complete-ultrametric-relative} implies \ref{item:intro-complete-ultrametric-compact}.

If $X$ is compact, every compatible ultrametric is complete, so
\[
 \UComp{X}{R}=\UMet{X}{R}.
\]
By \Cref{thm:introduction-ultrametric}, this space is Borel in $\UCPM{X}{R}$ and completely metrizable. If $\UComp{X}{R}$ is empty, the same two properties hold. Hence \ref{item:intro-complete-ultrametric-compact} implies the other three conditions. Finally, the definition of complete metrizability and \Cref{prop:ultrametric-existence}\ref{item:ultrametric-existence-complete} show that $\UComp{X}{R}$ is nonempty precisely when $X$ is completely metrizable. Thus the last condition is also equivalent to $X$ being either not completely metrizable or compact.
\end{proof}

We finish by describing $\UComp{X}{R}$ for noncharacteristic $R$ and by stating the compactness characterization for bounded complete compatible ultrametrics.

\begin{remark}\label{rmk:ucomp-range}
If $R$ is not characteristic, its positive elements are bounded away from zero. Every compatible $R$-valued ultrametric is then uniformly discrete and complete. Hence {$\UMet{X}{R}=\UComp{X}{R}$} \cite[Remark~1.5]{ishiki-baire}.
\end{remark}

For characteristic $R$, restricting to bounded ultrametrics leaves the compactness criterion unchanged.

\begin{corollary}\label{rmk:bounded-ucomp}
Let
$R$
be characteristic and let
$X$
be ultrametrizable.
Let
$\UCPMbd{X}{R}$
{denote the subspace of bounded elements of
$\UCPM{X}{R}$},
and let
$\UMetbd{X}{R}$
{denote the subspace of bounded elements of
$\UMet{X}{R}$}.
Put
\[
 \UCompbd{X}{R}=\UComp{X}{R}\cap\UMetbd{X}{R}.
\]
Then the following conditions are equivalent.
\begin{enumerate}[label=\textup{(V\arabic*)},ref=\textup{(V\arabic*)}]
\item\label{item:bounded-complete-ultrametric-ambient}
$\UCompbd{X}{R}$
is Borel in
$\UCPMbd{X}{R}$.
\item\label{item:bounded-complete-ultrametric-relative}
$\UCompbd{X}{R}$
is Borel relative to
$\UMetbd{X}{R}$.
\item\label{item:bounded-complete-ultrametric-complete}
$\UCompbd{X}{R}$
is completely metrizable in the topology induced by
$\umetdis_X^R$.
\item\label{item:bounded-complete-ultrametric-compact}
The space
$\UCompbd{X}{R}$
is empty or
$X$
is compact.
\item\label{item:bounded-complete-ultrametric-existence}
The space
$X$
is not completely metrizable or is compact.
\end{enumerate}
\end{corollary}

\begin{proof}
The argument of \Cref{thm:introduction-ucomp} {also applies to the subspaces consisting of bounded ultrametrics and pseudo-ultrametrics.} The required tree family in \Cref{thm:ucomplete-tree-reduction} is bounded, and {the extensor in \Cref{lem:ultrametric-max-transfer} sends the transported tree family to ultrametrics with a common upper bound in $R$. When $X$ is completely metrizable, it may also be chosen to preserve completeness by \eqref{eq:ultrametric-extension-completeness}.}
\end{proof}

\printbibliography

\end{document}